\documentclass[final]{siamltex}
\usepackage{amssymb}
\usepackage{amsmath}
\usepackage{amsfonts}
\usepackage{sw20siam}
\usepackage{version}
\usepackage{enumitem}

\newtheorem{remark}[theorem]{Remark}

\input{tcilatex}
\allowdisplaybreaks
\begin{document}

\title{On Explicit Recursive Formulas in the Spectral Perturbation Analysis of a Jordan Block\thanks{%
This work was supported by the Air Force Office of Scientific Research under
grant \#FA9550-08-1-0103.}}
\author{AARON WELTERS\thanks{%
Dept.\ of Mathematics, Univ.\ of California at Irvine, Irvine CA 92697 (%
\texttt{awelters@math.uci.edu}).}}
\maketitle

\begin{abstract}
Let $A\left( \varepsilon \right) $ be an analytic square matrix and $\lambda
_{0}$ an eigenvalue of $A\left( 0\right) $ of algebraic multiplicity $m\geq 1
$.  Then under the condition, $\frac{\partial }{\partial
\varepsilon }\det \left( \lambda I-A\left( \varepsilon \right) \right)
|_{\left( \varepsilon ,\lambda \right) =\left( 0,\lambda _{0}\right) }\neq 0$%
, we prove that the Jordan normal form of $A\left( 0\right) $ corresponding
to the eigenvalue $\lambda _{0}$ consists of a single $m\times m$ Jordan
block, the perturbed eigenvalues near $\lambda _{0}$ and their corresponding eigenvectors can be represented by a single convergent Puiseux series containing only powers of $\varepsilon ^{1/m}$, and there are explicit recursive formulas to compute all the Puiseux series coefficients from just the derivatives of $%
A\left( \varepsilon \right) $ at the origin.  Using these recursive
formulas we calculate the series coefficients up to the second order and
list them for quick reference.  This paper gives, under a generic
condition, explicit recursive formulas to compute the perturbed eigenvalues
and eigenvectors for non-selfadjoint analytic perturbations of matrices with
non-derogatory eigenvalues.
\end{abstract}

\keyphrases{Matrix Perturbation Theory, Degenerate Eigenvalue, Jordan Block, Perturbation of
Eigenvalues and Eigenvectors, Puiseux Series, Recursive Formula}
\AMclass{15A15, 15A18, 15A21, 41A58, 47A55, 47A56, 65F15, 65F40}

\pagestyle{myheadings} \thispagestyle{plain} 
\markboth{AARON
WELTERS}{RECURSIVE FORMULAS AND JORDAN BLOCK PERTURBATIONS}

\section{Introduction}

Consider an analytic square matrix $A\left( \varepsilon \right) $ and its
unperturbed matrix $A\left( 0\right) $ with a degenerate eigenvalue $\lambda
_{0}$.  A fundamental problem in the analytic perturbation theory of
non-selfadjoint matrices is the determination of the perturbed eigenvalues
near $\lambda _{0}$ along with their corresponding eigenvectors of the
matrix $A\left( \varepsilon \right) $ near $\varepsilon =0$.  More
specifically, let $A\left( \varepsilon \right) $ be a matrix-valued function
having a range in $%
\mathbb{C}
^{n\times n}$, the set of $n\times n$ matrices with complex entries, such
that its matrix elements are analytic functions of $\varepsilon $ in a
neighborhood of the origin. Let $\lambda _{0}$ be an eigenvalue of the
matrix $A\left( 0\right) $ with algebraic multiplicity $m\geq 1$.  Then in this situation, it is well known 
\cite[\S 6.1.7]{Baumgartel1985}, \cite[\S II.1.8]{Kato1995} that for sufficiently
small $\varepsilon $ all the perturbed eigenvalues near $\lambda _{0}$,
called the $\lambda _{0}$-group, and their corresponding eigenvectors may be
represented as a collection of convergent Puiseux series, i.e., convergent Taylor series in a fractional power of $\varepsilon $.  What is not well known,
however, is how we compute these Puiseux series when $A\left( \varepsilon
\right) $ is a non-selfadjoint analytic perturbation and $\lambda _{0}$ is a
defective eigenvalue of $A\left( 0\right)$.  There are sources on the
subject like \cite[\S 7.4]{Baumgartel1985}, \cite{Lidskii1966}, \cite{MoroBurkeOverton1997}%
, \cite[\S 32]{VainbergTrenogin1974}, and \cite{VishikLjusternik1960} but it was
found that there lacked explicit formulas, recursive or otherwise, to
compute the series coefficients beyond the first order terms.  Thus the
fundamental problem that this paper addresses is to find
explicit recursive formulas to determine the Puiseux series coefficients for
the $\lambda _{0}$-group and their eigenvectors.

This problem is of applied and theoretic importance, for example, in studying the spectral properties of dispersive media such as photonic crystals.  In particular, this is especially true in the study of slow light \cite{FigotinVitebskiy2006}--\cite{YargaSertelVolakis2008}, where the characteristic equation, $\det \left( \lambda I-A\left(\varepsilon \right) \right)=0$, represents implicitly the dispersion relation for Bloch waves in the periodic crystal.  In this setting $\varepsilon$ represents a small change in frequency, $A(\varepsilon)$ is the Transfer matrix of a unit cell, and its eigenpairs, $(\lambda(\varepsilon),x(\varepsilon))$, correspond to the Bloch waves.  From a practical and theoretical point of view, condition \eqref{generic condition} on the dispersion relation or its equivalent formulation in Theorem \ref{Generic Condition Theorem}.ii of this paper regarding the group velocity for this setting, arises naturally in the study of slow light where the Jordan normal form of the unperturbed Transfer matrix, $A(0)$, and the perturbation expansions of the eigenpairs of the Transfer matrix play a central role in the analysis of slow light waves.

\subsection*{Main Results}%

In this paper under the \textit{generic condition},%
\begin{equation}
\frac{\partial }{\partial \varepsilon }\det \left( \lambda I-A\left(
\varepsilon \right) \right) \big |_{\left( \varepsilon ,\lambda \right)
=\left( 0,\lambda _{0}\right) }\neq 0\text{,}  \label{generic condition}
\end{equation}%
we show that $\lambda_0$ is a non-derogatory eigenvalue of $A(0)$ and the fundamental problem mentioned above can be solved.  In particular, we
prove Theorem \ref{Generic Condition Theorem} and Theorem \ref{Main Results
Theorem} which together state that when condition $\eqref{generic condition}$ is true then the
Jordan normal form of $A\left( 0\right) $ corresponding to the eigenvalue $%
\lambda _{0}$ consists of a single $m\times m$ Jordan block, the $\lambda _{0}$%
-group and their corresponding eigenvectors can each be represented by a single
convergent Puiseux series whose branches are given by
\begin{eqnarray*}
\lambda _{h}\left( \varepsilon \right) &=&\lambda
_{0}+\sum\limits_{k=1}^{\infty }\alpha _{k}\left( \zeta ^{h}\varepsilon ^{%
\frac{1}{m}}\right) ^{k} \\
x_{h}\left( \varepsilon \right) &=&\beta_{0}+\sum\limits_{k=1}^{\infty
}\beta_{k}\left( \zeta ^{h}\varepsilon ^{\frac{1}{m}}\right) ^{k}
\end{eqnarray*}%
for $h=0, \ldots ,m-1$ and any fixed branch of $\varepsilon^{\frac{1}{m}}$, where $\zeta =e^{\frac{2\pi}{m}i}\text{, }\{\alpha _{k}\}_{k=1}^\infty\subseteq\mathbb{C}\text{, }\{\beta _{k}\}_{k=0}^\infty\subseteq\mathbb{C}^{n\times 1}\text{, }\alpha
_{1}\neq 0$, and $\beta_{0}$ is an eigenvector of $A\left( 0\right) $
corresponding to the eigenvalue $\lambda _{0}$.  More importantly though, Theorem \ref{MainTheorem} gives explicit recursive formulas that allows us to determine
the Puiseux series coefficients, $\{\alpha _{k}\}_{k=1}^\infty$ and $\{\beta _{k}\}_{k=0}^\infty$, from just the derivatives of $A\left(
\varepsilon \right) $ at $\varepsilon =0$.  Using these recursive formulas,
we compute the leading Puiseux series coefficients up to the second order
and list them in Corollary \ref{Second order coefficients corollary}.

The key to all of our results is the study of the characteristic equation
for the analytic matrix $A\left( \varepsilon \right) $ under the generic
condition (\ref{generic condition}).  By an application of the implicit
function theorem, we are able to derive the functional relation between the
eigenvalues and the perturbation parameter.  This leads to the implication
that the Jordan normal form of the unperturbed matrix $A\left( 0\right) $
corresponding to the eigenvalue $\lambda _{0}$ is a single $m\times m$
Jordan block.  From this, we are able to use the method of undetermined
coefficients along with a careful combinatorial analysis to get explicit recursive formulas for determining the Puiseux series coefficients.

We want to take a moment here to show how the results of this paper can be
used to determine the Puiseux series coefficients up to the second order for
the case in which the non-derogatory eigenvalue $\lambda_0$ has algebraic multiplicity $m\geq 2$.  We start by putting $A\left( 0\right) $ into the Jordan normal form \cite[\S 6.5: The
Jordan Theorem]{LancasterTismenetsky1985} 
\begin{equation}
U^{-1}A\left( 0\right) U=\left[ 
\begin{array}{c|c}
J_{m}\left( \lambda _{0}\right) &  \\ \hline
& W_{0}%
\end{array}%
\right] \text{,}
\end{equation}%
where (see notations at end of \S 1) $J_{m}\left( \lambda _{0}\right) $ is
an $m\times m$ Jordan block corresponding to the eigenvalue $\lambda _{0}$\
and $W_{0}$ is the Jordan normal form for the rest of the spectrum.  Next, define the vectors $u_{1}$,\ldots , $u_{m}$, as the first $m$ columns
of the matrix $U$, 
\begin{equation}
u_{i}:=Ue_{i}\text{, }1\leq i\leq m
\end{equation}%
(These vectors have the properties that $u_{1}$ is an eigenvector of $%
A\left( 0\right) $ corresponding to the eigenvalue $\lambda _{0}$, they form
a Jordan chain with generator $u_{m}$, and are a basis for the algebraic
eigenspace of $A\left( 0\right) $ corresponding to the eigenvalue $\lambda
_{0}$).  We then partition the matrix $U^{-1}A^\prime(0)
U $ conformally to the blocks $J_{m}\left( \lambda _{0}\right) $ and $W_{0}$
of the matrix $U^{-1}A\left( 0\right) U$ as such 
\begin{equation}
U^{-1}A^\prime(0)U=\left[ 
\begin{array}{ccccc|ccc}
\ast & \ast & \ast & \cdots & \ast & \ast & \cdots & \ast \\ 
\vdots & \vdots & \vdots & \ddots & \vdots & \vdots & \ddots & \vdots \\ 
\ast & \ast & \ast & \cdots & \ast & \ast & \cdots & \ast \\ 
a_{m-1,1} & \ast & \ast & \cdots & \ast & \ast & \cdots & \ast \\ 
a_{m,1} & a_{m,2} & \ast & \cdots & \ast & \ast & \cdots & \ast \\ \hline
\ast & \ast & \ast & \cdots & \ast & \ast & \cdots & \ast \\ 
\vdots & \vdots & \vdots & \ddots & \vdots & \vdots & \ddots & \vdots \\* 
\ast & \ast & \ast & \cdots & \ast & \ast & \cdots & \ast%
\end{array}%
\right] \text{.}
\end{equation}%
Now, by Theorem \ref{Generic Condition Theorem} and Theorem \ref%
{Main Results Theorem}, it follows that 
\begin{equation}
a_{m,1}=-\frac{\frac{\partial }{\partial \varepsilon }\det \left( \lambda
I-A\left( \varepsilon \right) \right) |_{\left( \varepsilon ,\lambda \right)
=\left( 0,\lambda _{0}\right) }}{\left( \frac{\frac{\partial ^{m}}{\partial
\lambda ^{m}}\det \left( \lambda I-A\left( \varepsilon \right) \right)
|_{\left( \varepsilon ,\lambda \right) =\left( 0,\lambda _{0}\right) }}{m!}%
\right) }\text{.}
\end{equation}%
And hence the generic condition is true if and only if $a_{m,1}\not=0$.  This gives us an alternative method to determine whether the generic
condition (\ref{generic condition}) is true or not.

Lets now assume that $a_{m,1}\not=0$ and hence that the generic condition is
true.  Define $f\left( \varepsilon ,\lambda \right) :=\det
\left( \lambda I-A\left( \varepsilon \right) \right) $.  Then by Theorem %
\ref{Main Results Theorem} and Corollary \ref{Second order coefficients
corollary} there is exactly one convergent Puiseux series for the perturbed eigenvalues near $\lambda _{0}$ and one for their corresponding eigenvectors whose branches are given by%
\begin{eqnarray}
\lambda _{h}\left( \varepsilon \right) &=&\lambda _{0}+\alpha _{1}\left(
\zeta ^{h}\varepsilon ^{\frac{1}{m}}\right) +\alpha _{2}\left( \zeta
^{h}\varepsilon ^{\frac{1}{m}}\right) ^{2}+\sum\limits_{k=3}^{\infty
}\alpha _{k}\left( \zeta ^{h}\varepsilon ^{\frac{1}{m}}\right) ^{k} \\
x_{h}\left( \varepsilon \right) &=&x_{0}+\beta_{1}\left( \zeta ^{h}\varepsilon ^{%
\frac{1}{m}}\right) +\beta_{2}\left( \zeta ^{h}\varepsilon ^{\frac{1}{m}}\right)
^{2}+\sum\limits_{k=3}^{\infty }\beta_{k}\left( \zeta ^{h}\varepsilon ^{\frac{1}{%
m}}\right) ^{k}
\end{eqnarray}%
for $h=0, \ldots ,m-1$ and any fixed branch of $\varepsilon^{\frac{1}{m}}$, where $\zeta =e^{\frac{2\pi }{m}i}$.  Furthermore,
the series coefficients up to second order may be given by 
\begin{eqnarray}
\alpha _{1} &=&a_{m,1}^{1/m}=\left( -\frac{\frac{\partial f}{\partial
\varepsilon }\left( 0,\lambda _{0}\right) }{\frac{1}{m!}\frac{\partial ^{m}f%
}{\partial \lambda ^{m}}\left( 0,\lambda _{0}\right) }\right) ^{1/m}\neq 0%
\text{,} \\
\alpha _{2} &=&\frac{a_{m-1,1}+a_{m,2}}{m\alpha _{1}^{m-2}}=\frac{-\left(
\alpha _{1}^{m+1}\frac{1}{\left( m+1\right) !}\frac{\partial ^{m+1}f}{%
\partial \lambda ^{m+1}}\left( 0,\lambda _{0}\right) +\alpha _{1}\frac{%
\partial ^{2}f}{\partial \lambda \partial \varepsilon }\left( 0,\lambda
_{0}\right) \right) }{m\alpha _{1}^{m-1}\left( \frac{1}{m!}\frac{\partial
^{m}f}{\partial \lambda ^{m}}\left( 0,\lambda _{0}\right) \right) }\text{,}
\\
\beta_{0} &=&u_{1}\text{, }\beta_{1}=\alpha _{1}u_{2}\text{, }\beta_{2} =\left\{ 
\begin{array}{c}
-\Lambda A^\prime(0) u_{1}+\alpha _{2}u_{2}\text{, if }m=2
\\ 
\alpha _{2}u_{2}+\alpha _{1}^{2}u_{3}\text{, if }m>2%
\end{array}%
\right.
\end{eqnarray}%
for any choice of the $m$th root of $a_{m,1}$ and where $\Lambda $ is given
in (\ref{PartialInverseOfTheJordanNormalFormOfANot}).

The explicit recursive formulas for computing higher order terms, $\alpha
_{k}, \beta_{k}$, are given by (%
\ref{MainResultsTheoremRecursiveFormulas2}) and (\ref%
{MainResultsTheoremRecursiveFormulas3}) in Theorem \ref{Main Results Theorem}.  The steps which should be used to determine these higher order terms are discussed in Remark \ref{RemarkSteps} and an example showing how to calculating $\alpha_3,\beta_3$ using these steps, when $m\geq 3$, is provided. 

\subsection*{Example}%

The following example may help to give a better idea of these results. \
Consider 
\begin{equation}
A\left( \varepsilon \right) :=\left[ 
\begin{array}{rrr}
-\frac{1}{2} & 1 & \frac{1}{2} \\ 
\frac{1}{2} & 0 & -\frac{1}{2} \\ 
-1 & 1 & 1%
\end{array}%
\right] +\varepsilon \left[ 
\begin{array}{rrr}
2 & 0 & -1 \\ 
2 & 0 & -1 \\ 
1 & 0 & 0%
\end{array}%
\right] \text{.}  \label{Example1}
\end{equation}

Here $\lambda _{0}=0$ is a non-derogatory eigenvalue of $A\left( 0\right) $ of algebraic
multiplicity $m=2$.  We put $A\left( 0\right) $ into the Jordan normal form%
\begin{equation*}
U^{-1}A\left( 0\right) U=\left[ 
\begin{array}{cc|c}
0 & 1 & 0 \\ 
0 & 0 & 0 \\ \hline
0 & 0 & 1/2%
\end{array}%
\right] \text{, }U=%
\begin{bmatrix}
1 & 1 & 1 \\ 
0 & 1 & 1 \\ 
1 & 1 & 0%
\end{bmatrix}%
\text{, }U^{-1}=\left[ 
\begin{array}{rrr}
1 & -1 & 0 \\ 
-1 & 1 & 1 \\ 
1 & 0 & -1%
\end{array}%
\right] \text{,}
\end{equation*}%
so that $W_{0}=1/2$.  We next define the vectors $u_{1}$, $u_{2}$, as the
first two columns of the matrix $U$, 
\begin{equation*}
u_{1}:=%
\begin{bmatrix}
1 \\ 
0 \\ 
1%
\end{bmatrix}%
\text{, }u_{2}:=%
\begin{bmatrix}
1 \\ 
1 \\ 
1%
\end{bmatrix}%
\text{.}
\end{equation*}%
Next we partition the matrix $U^{-1}A^\prime(0) U$
conformally to the blocks $J_{m}\left( \lambda _{0}\right) $ and $W_{0}$ of
the matrix $U^{-1}A\left( 0\right) U$ as such%
\begin{equation*}
U^{-1}A^\prime(0)U=\left[ 
\begin{array}{rrr}
1 & -1 & 0 \\ 
-1 & 1 & 1 \\ 
1 & 0 & -1%
\end{array}%
\right] \left[ 
\begin{array}{rrr}
2 & 0 & -1 \\ 
2 & 0 & -1 \\ 
1 & 0 & 0%
\end{array}%
\right] 
\begin{bmatrix}
1 & 1 & 1 \\ 
0 & 1 & 1 \\ 
1 & 1 & 0%
\end{bmatrix}%
=\left[ 
\begin{array}{cc|c}
0 & \ast & \ast \\ 
1 & 1 & \ast \\ \hline
\ast & \ast & \ast%
\end{array}%
\right] \text{.}
\end{equation*}%
Here $a_{2,1}=1$, $a_{1,1}=0$, and $a_{2,2}=1$.  Then%
\begin{equation*}
1=a_{2,1}=-\frac{\frac{\partial }{\partial \varepsilon }\det \left( \lambda
I-A\left( \varepsilon \right) \right) |_{\left( \varepsilon ,\lambda \right)
=\left( 0,\lambda _{0}\right) }}{\left( \frac{\frac{\partial ^{2}}{\partial
\lambda ^{2}}\det \left( \lambda I-A\left( \varepsilon \right) \right)
|_{\left( \varepsilon ,\lambda \right) =\left( 0,\lambda _{0}\right) }}{2!}%
\right) }\text{,}
\end{equation*}%
implying that the generic condition (\ref{generic condition}) is true.  Define $f\left( \varepsilon ,\lambda \right) :=\det \left( \lambda I-A\left(
\varepsilon \right) \right) =\lambda ^{3}-2\lambda ^{2}\varepsilon -\frac{1}{%
2}\lambda ^{2}+\lambda \varepsilon ^{2}-\frac{1}{2}\lambda \varepsilon
+\varepsilon ^{2}+\frac{1}{2}\varepsilon $.  Then there is exactly one convergent Puiseux series for the perturbed eigenvalues near $\lambda _{0}=0$ and one for their corresponding eigenvectors whose branches are given by%
\begin{eqnarray*}
\lambda _{h}\left( \varepsilon \right) &=&\lambda _{0}+\alpha _{1}\left(
\left( -1\right) ^{h}\varepsilon ^{\frac{1}{2}}\right) +\alpha _{2}\left(
\left( -1\right) ^{h}\varepsilon ^{\frac{1}{2}}\right)
^{2}+\sum\limits_{k=3}^{\infty }\alpha _{k}\left( \left( -1\right)
^{h}\varepsilon ^{\frac{1}{2}}\right) ^{k} \\
x_{h}\left( \varepsilon \right) &=&\beta_{0}+\beta_{1}\left( \left( -1\right)
^{h}\varepsilon ^{\frac{1}{2}}\right) +\beta_{2}\left( \left( -1\right)
^{h}\varepsilon ^{\frac{1}{2}}\right) ^{2}+\sum\limits_{k=3}^{\infty
}\beta_{k}\left( \left( -1\right) ^{h}\varepsilon ^{\frac{1}{2}}\right) ^{k}
\end{eqnarray*}%
for $h=0, 1$ and any fixed branch of $\varepsilon^{\frac{1}{2}}$. \ Furthermore, the series coefficients up to second order
may be given by 
\begin{eqnarray*}
\alpha _{1} &=&1=\sqrt{1}=\sqrt{a_{2,1}}=\sqrt{\left( -\frac{\frac{\partial
f}{\partial \varepsilon }\left( 0,\lambda _{0}\right) }{\frac{1}{2!}\frac{%
\partial ^{2}f}{\partial \lambda ^{2}}\left( 0,\lambda _{0}\right) }\right) }%
\neq 0\text{,} \\
\alpha _{2} &=&\frac{1}{2}=\frac{a_{1,1}+a_{2,2}}{2}=\frac{-\left( \alpha
_{1}^{3}\frac{1}{3!}\frac{\partial ^{3}f}{\partial \lambda ^{3}}\left(
0,\lambda _{0}\right) +\alpha _{1}\frac{\partial ^{2}f}{\partial \lambda
\partial \varepsilon }\left( 0,\lambda _{0}\right) \right) }{\alpha
_{1}\left( \frac{1}{2!}\frac{\partial ^{2}f}{\partial \lambda ^{2}}\left(
0,\lambda _{0}\right) \right) }\text{,} \\
\beta_{0} &=&%
\begin{bmatrix}
1 \\ 
0 \\ 
1%
\end{bmatrix}%
\text{, }\beta_{1}=%
\begin{bmatrix}
1 \\ 
1 \\ 
1%
\end{bmatrix}%
\text{, }
\beta_{2} =-\Lambda A^\prime(0) u_{1}+\alpha _{2}u_{2}
\end{eqnarray*}%
by choosing the positive square root of $a_{2,1}=1$ and where $\Lambda $ is
given in (\ref{PartialInverseOfTheJordanNormalFormOfANot}).  Here%
\begin{eqnarray*}
\Lambda &=&U\left[ 
\begin{array}{c|c}
J_{m}\left( 0\right) ^{\ast } &  \\ \hline
& \left( W_{0}-\lambda _{0}I_{n-m}\right) ^{-1}%
\end{array}%
\right] U^{-1} \\
&=&%
\begin{bmatrix}
1 & 1 & 1 \\ 
0 & 1 & 1 \\ 
1 & 1 & 0%
\end{bmatrix}%
\left[ 
\begin{array}{cc|c}
0 & 0 & 0 \\ 
1 & 0 & 0 \\ \hline
0 & 0 & \left( 1/2\right) ^{-1}%
\end{array}%
\right] \left[ 
\begin{array}{rrr}
1 & -1 & 0 \\ 
-1 & 1 & 1 \\ 
1 & 0 & -1%
\end{array}%
\right] =\left[ 
\begin{array}{rrr}
3 & -1 & -2 \\ 
3 & -1 & -2 \\ 
1 & -1 & 0%
\end{array}%
\right]\\
\beta_{2} &=&-\Lambda A^\prime(0) u_{1}+\alpha _{2}u_{2} \\
&=&-\left[ 
\begin{array}{rrr}
3 & -1 & -2 \\ 
3 & -1 & -2 \\ 
1 & -1 & 0%
\end{array}%
\right] \allowbreak \left[ 
\begin{array}{rrr}
2 & 0 & -1 \\ 
2 & 0 & -1 \\ 
1 & 0 & 0%
\end{array}%
\right] 
\begin{bmatrix}
1 \\ 
0 \\ 
1%
\end{bmatrix}%
+\frac{1}{2}%
\begin{bmatrix}
1 \\ 
1 \\ 
1%
\end{bmatrix}%
=\frac{1}{2}%
\begin{bmatrix}
1 \\ 
1 \\ 
1%
\end{bmatrix}%
\text{.}
\end{eqnarray*}%

Now compare this to the actual perturbed eigenvalues of our example (\ref%
{Example1}) near $\lambda _{0}=0$ and their corresponding eigenvectors%
\begin{eqnarray*}
\lambda _{h}\left( \varepsilon \right) &=&\frac{1}{2}\varepsilon +\left(
-1\right) ^{h}\frac{1}{2}\varepsilon ^{\frac{1}{2}}\left( \varepsilon
+4\right) ^{\frac{1}{2}} \\
&=&\left( \left( -1\right) ^{h}\varepsilon ^{\frac{1}{2}}\right) +\frac{1}{2}%
\left( \left( -1\right) ^{h}\varepsilon ^{\frac{1}{2}}\right)
^{2}+\sum\limits_{k=3}^{\infty }\alpha _{k}\left( \left( -1\right)
^{h}\varepsilon ^{\frac{1}{2}}\right) ^{k} \\
x_{h}\left( \varepsilon \right) &=&%
\begin{bmatrix}
1 \\ 
0 \\ 
1%
\end{bmatrix}%
+%
\begin{bmatrix}
1 \\ 
1 \\ 
1%
\end{bmatrix}%
\lambda _{h}\left( \varepsilon \right) \\
&=&%
\begin{bmatrix}
1 \\ 
0 \\ 
1%
\end{bmatrix}%
+%
\begin{bmatrix}
1 \\ 
1 \\ 
1%
\end{bmatrix}%
\left( \left( -1\right) ^{h}\varepsilon ^{\frac{1}{2}}\right) +\frac{1}{2}%
\begin{bmatrix}
1 \\ 
1 \\ 
1%
\end{bmatrix}%
\left( \left( -1\right) ^{h}\varepsilon ^{\frac{1}{2}}\right)
^{2}+\sum\limits_{k=3}^{\infty }\beta_{k}\left( \left( -1\right) ^{h}\varepsilon
^{\frac{1}{2}}\right) ^{k}
\end{eqnarray*}%
for $h=0,1$ and any fixed branch of $\varepsilon^{\frac{1}{2}}$.  We see that indeed our formulas for the Puiseux series
coefficients are correct up to the second order.

\subsection*{Comparison to Known Results}%

There is a fairly large amount of literature on eigenpair perturbation expansions for analytic perturbations of non-selfadjoint matrices with degenerate eigenvalues 
(e.g.\ \cite{Baumgartel1985}--\cite{VishikLjusternik1960}, \cite{AndrewChuLancaster1993}--\cite{Sun1990}).  However, most of the literature (e.g.\ \cite{Lidskii1966}, \cite{MoroBurkeOverton1997}, \cite{AndrewChuLancaster1993}, \cite{AndrewTan2000}, \cite{HrynivLancaster1999}, 
\cite{Lancaster1964}--\cite{MoroDopico2001}, \cite{Sun1985}, \cite{Sun1990}) contains results only on the first order expansions of the Puiseux series or considers higher order terms only in the case of simple or semisimple eigenvalues.  For those works that do address higher order terms for defective eigenvalues (e.g.\ \cite{Baumgartel1985}, \cite{Kato1995}, \cite{VainbergTrenogin1974}, \cite{VishikLjusternik1960}, \cite{Chu1990}, \cite{JeannerodPflugel1999}, \cite{SeyranianMailybaev2003}), it was found that there did not exist explicit recursive formulas for all the Puiseux coefficients when the matrix perturbations were non-linear.  One of the purposes and achievements of this paper are the explicit recursive formulas \eqref{MainResultsTheoremRecursiveFormulas1}--\eqref{MainResultsTheoremRecursiveFormulas3} in Theorem \ref{MainTheorem} which give all the higher order terms in the important case of degenerate eigenvalues which are non-derogatory, that is, the case in which a degenerate eigenvalue of the unperturbed matrix has a single Jordan block for its corresponding Jordan structure.
Our theorem generalizes and extends the results of \cite[pp.\ 315--317, (4.96) \& (4.97)]{Baumgartel1985}, \cite[pp.\ 415--418]{VainbergTrenogin1974}, and \cite[pp.\ 17--20]{VishikLjusternik1960} to non-linear analytic matrix perturbations and makes explicit the recursive formulas for calculating the perturbed eigenpair Puiseux expansions.  Furthermore, in Proposition \ref{ExplicitRecursiveFormulasPolynomials} we give an explicit recursive formula for calculating the polynomials $\left\{r_l\right\}_{l\in\mathbb{N}}$.  These polynomials must be calculated in order to determine the higher order terms in the eigenpair Puiseux series expansions (see \eqref{MainResultsTheoremRecursiveFormulas3} in Theorem \ref{MainTheorem} and Remark \ref{RemarkSteps}).  These polynomials appear in \cite[p.\ 315, (4.95)]{Baumgartel1985}, \cite[p.\ 414, (32.24)]{VainbergTrenogin1974}, and \cite[p.\ 19, (34)]{VishikLjusternik1960} under different notation (compare with Proposition \ref{eq:B.1.ii}.\ref{PropB1ii}) but no method is given to calculate them.  As such, Proposition \ref{ExplicitRecursiveFormulasPolynomials} is an important contribution in the explicit recursive calculation of the higher order terms in the eigenpair Puiseux series expansions.

Another purpose of this paper is to give, in the case of degenerate non-derogatory eigenvalues, an easily accessible and quickly referenced list of first and second order terms for the Puiseux series expansions of the perturbed eigenpairs.  When the generic condition \eqref{generic condition} is satisfied, Corollary \ref{2ndOrderCoeffCorollary} gives this list.  Now for first order terms there are quite a few papers on formulas for determining them, see for example \cite{MoroDopico2001} which gives a good survey of first order perturbation theory.  But for second order terms, it was difficult to find any results in the literature similar to and as explicit as Corollary \ref{2ndOrderCoeffCorollary} for the case of degenerate non-derogatory eigenvalues with arbitrary algebraic multiplicity and non-linear analytic perturbations.  Results comparable to ours can be found in \cite[p.\ 316]{Baumgartel1985}, \cite[pp.\ 415--418]{VainbergTrenogin1974}, \cite[pp.\ 17-20]{VishikLjusternik1960}, and \cite[pp.\ 37--38, 50--54, 125--128]{SeyranianMailybaev2003}, although it should be noted that in \cite[p.\ 417]{VainbergTrenogin1974} the formula for the second order term of the perturbed eigenvalues contains a misprint.

\subsection*{Overview}%

Section 2 deals with the generic condition (\ref{generic condition}).  We
give conditions that are equivalent to the generic condition in Theorem \ref%
{Generic Condition Theorem}.  In \S 3 we give the main results of this
paper in Theorem \ref{Main Results Theorem}, on the determination of the
Puiseux series with the explicit recursive formulas for calculating the series
coefficients.  As a corollary we give the exact leading order terms, up to the second order, for the Puiseux series coefficients.  Section 4 contains the proofs of the results in \S 2 and \S 3.

\subsection*{Notation}%

Let $%
\mathbb{C}
^{n\times n}$ be the set of all $n\times n$ matrices with complex entries
and $%
\mathbb{C}
^{n\times 1}$ the set of all $n\times 1$ column vectors with complex
entries.  For $a\in 
\mathbb{C}
$, $A\in 
\mathbb{C}
^{n\times n}$, and $x=\left[ a_{i,1}\right] _{i=1}^{n}\in 
\mathbb{C}
^{n\times 1}$ we denote by $a^{\ast }$, $A^{\ast }$, and $x^{\ast }$, the
complex conjugate of $a$, the conjugate transpose of $A$, and the $1\times n$
row vector $x^{\ast }:=\left[ 
\begin{array}{rrr}
a_{1,1}^{\ast } & \cdots & a_{n,1}^{\ast }%
\end{array}%
\right] $.  For $x$, $y\in 
\mathbb{C}
^{n\times 1}$ we let $\left( x,y\right) :=x^{\ast }y$ be the standard inner
product.  The matrix $I\in 
\mathbb{C}
^{n\times n}$ is the identity matrix and its $j$th column is $e_{j}\in 
\mathbb{C}
^{n\times 1}$.  The matrix $I_{n-m}$ is the $\left( n-m\right) \times
\left( n-m\right) $ identity matrix.  Define an $m\times m$ Jordan block
with eigenvalue $\lambda $ to be 
\begin{equation*}
J_{m}\left( \lambda \right) :=%
\begin{bmatrix}
\lambda & 1 &  &  &  \\ 
& . & . &  &  \\ 
&  & . & . &  \\ 
&  &  & . & 1 \\ 
&  &  &  & \lambda%
\end{bmatrix}%
\text{.}
\end{equation*}%
When the matrix $A\left( \varepsilon \right) \in 
\mathbb{C}
^{n\times n}$ is analytic at $\varepsilon =0$ we define $A^\prime(0):=\frac{dA}{d\varepsilon}(0)$ and $A_k:=\frac{1}{k!}\frac{d^k A}{d\varepsilon^k}(0)$.  Let $\zeta :=e^{i\frac{2\pi }{m}}$.

\section{The Generic Condition}

The following theorem, which is proved in \S 4, gives conditions which are
equivalent to the generic one (\ref{generic condition}).

\begin{theorem}
\label{Generic Condition Theorem}Let $A\left( \varepsilon \right) $ be a
matrix-valued function having a range in $%
\mathbb{C}
^{n\times n}$ such that its matrix elements are analytic functions of $%
\varepsilon $ in a neighborhood of the origin.  Let $\lambda _{0}$ be an
eigenvalue of the unperturbed matrix $A\left( 0\right) $ and denote by $m$
its algebraic multiplicity.  Then the following statements are equivalent:

\begin{romannum}
\item The characteristic polynomial $\det \left( \lambda I-A\left(
\varepsilon \right) \right) $ has a simple zero with respect to $\varepsilon 
$ at $\lambda =\lambda _{0}$ and $\varepsilon =0$, i.e., 
\begin{equation*}
\frac{\partial }{\partial \varepsilon }\det \left( \lambda I-A\left(
\varepsilon \right) \right) \big |_{\left( \varepsilon ,\lambda \right)
=\left( 0,\lambda _{0}\right) }\neq 0\text{.}
\end{equation*}

\item The characteristic equation, $\det (\lambda I-A\left( \varepsilon
\right) )=0$, has a unique solution, $\varepsilon \left( \lambda \right) $,
in a neighborhood of $\lambda =\lambda _{0}$ with $\varepsilon \left(
\lambda _{0}\right) =0$.  This solution is an analytic function with a zero
of order $m$ at $\lambda =\lambda _{0}$, i.e., 
\begin{equation*}
\frac{d^{0}\varepsilon \left( \lambda \right) }{d\lambda ^{0}}\Big |%
_{\lambda =\lambda _{0}}=\cdots =\frac{d^{m-1}\varepsilon \left( \lambda
\right) }{d\lambda ^{m-1}}\Big |_{\lambda =\lambda _{0}}=0\text{, }\frac{%
d^{m}\varepsilon \left( \lambda \right) }{d\lambda ^{m}}\Big |_{\lambda
=\lambda _{0}}\neq 0\text{.}
\end{equation*}

\item There exists a convergent Puiseux series whose branches are given by%
\begin{equation*}
\lambda _{h}\left( \varepsilon \right) =\lambda _{0}+\alpha _{1}\zeta
^{h}\varepsilon ^{\frac{1}{m}}+\sum\limits_{k=2}^{\infty }\alpha
_{k}\left( \zeta ^{h}\varepsilon ^{\frac{1}{m}}\right)^{k}\text{, }%
\end{equation*}%
for $h=0, \ldots ,m-1$ and any fixed branch of $\varepsilon^{\frac{1}{m}}$, where $\zeta =e^{\frac{2\pi }{m}i}$, such that the values of the branches give all the solutions of the
characteristic equation, for sufficiently small $\varepsilon $ and $\lambda $
sufficiently near $\lambda _{0}$.  Furthermore, the first order term is
nonzero, i.e., 
\begin{equation*}
\alpha _{1}\neq 0\text{.}
\end{equation*}

\item The Jordan normal form of $A\left( 0\right) $ corresponding to the
eigenvalue $\lambda _{0}$ consists of a single $m\times m$ Jordan block and
there exists an eigenvector $u_{0}$ of $A\left( 0\right) $ corresponding to
the eigenvalue $\lambda _{0}$ and an eigenvector $v_{0}$ of $A\left(
0\right) ^{\ast }$ corresponding to the eigenvalue $\lambda _{0}^{\ast }$
such that 
\begin{equation*}
\left( v_{0},A^{\prime}(0) u_{0}\right) \neq 0\text{.}
\end{equation*}
\end{romannum}
\end{theorem}
\section{Determination of the Puiseux Series and the Explicit Recursive Formulas for Calculating the Series}

This section contains the main results of this paper presented below in
Theorem \ref{Main Results Theorem}.  To begin we give some preliminaries
that are needed to set up the theorem.  Suppose that $A\left( \varepsilon
\right) $ is a matrix-valued function having a range in $%
\mathbb{C}
^{n\times n}$ with matrix elements that are analytic functions of $%
\varepsilon $ in a neighborhood of the origin and $\lambda _{0}$ is an
eigenvalue of the unperturbed matrix $A\left( 0\right) $ with algebraic
multiplicity $m$.  Assume that the generic condition%
\begin{equation*}
\frac{\partial }{\partial \varepsilon }\det \left( \lambda I-A\left(
\varepsilon \right) \right) \big |_{\left( \varepsilon ,\lambda \right)
=\left( 0,\lambda _{0}\right) }\neq 0\text{,}
\end{equation*}
is true.

Now, by these assumptions, we may appeal to Theorem \ref{Generic Condition
Theorem}.iv and conclude that the Jordan canonical form of $A(0)$ has
only one $m\times m$ Jordan block associated with $\lambda _{0}$.  Hence
there exists a invertible matrix $U$ $\in 
\mathbb{C}
^{n\times n}$ such that 
\begin{equation}
U^{-1}A\left( 0\right) U=\left[ 
\begin{array}{c|c}
J_{m}\left( \lambda _{0}\right) &  \\ \hline
& W_{0}%
\end{array}%
\right] \text{,}  \label{TheJordanNormalFormOfANot}
\end{equation}%
where $W_{0}$ is a $\left( n-m\right) \times \left( n-m\right) $ matrix such
that $\lambda _{0}$ is not one of its eigenvalues \cite[\S 6.5: The Jordan
Theorem]{LancasterTismenetsky1985}.

We define the vectors $u_{1},\ldots , u_{m}$, $v_{1},\ldots , v_{m}\in 
\mathbb{C}
^{n\times 1}$ as the first $m$ columns of the matrix $U$ and $\left(
U^{-1}\right) ^{\ast }$, respectively, i.e., 
\begin{eqnarray}
u_{i} &:&=Ue_{i}\text{, }1\leq i\leq m\text{,}  \label{BiorthogonalVectors}
\\
v_{i} &:&=\left( U^{-1}\right) ^{\ast }e_{i}\text{, }1\leq i\leq m\text{.} 
\label{BiorthogonalVectors_v}
\end{eqnarray}%
And define the matrix $\Lambda \in 
\mathbb{C}
^{n\times n}$ by%
\begin{equation}
\Lambda :=U\left[ 
\begin{array}{c|c}
J_{m}\left( 0\right) ^{\ast } &  \\ \hline
& \left( W_{0}-\lambda _{0}I_{n-m}\right) ^{-1}%
\end{array}%
\right] U^{-1}\text{,}  \label{PartialInverseOfTheJordanNormalFormOfANot}
\end{equation}%
where $\left( W_{0}-\lambda _{0}I_{n-m}\right) ^{-1}$ exists since $\lambda
_{0}$ is not an eigenvalue of $W_{0}$(for the important properties of the matrix $\Lambda$ see Appendix A).

Next, we introduce the polynomials $p_{j,i}=p_{j,i}\left( \alpha _{1}\text{%
, \ldots , }\alpha _{j-i+1}\right) $ in $\alpha _{1}$,\ldots , $\alpha
_{j-i+1}$, for $j\geq i\geq 0$, as the expressions%
\begin{eqnarray}
\left. 
\begin{array}{c}
p_{0,0}:=1\text{, }p_{j,0}:=0\text{, for }j>0\text{,} \\ 
p_{j,i}\left( \alpha _{1}\text{, \ldots , }\alpha _{j-i+1}\right) :=\dsum\limits_{\substack{ s_{1}+\cdots +s_{i}=j \\ 1\leq s_{\varrho }\leq j-i+1}}%
\displaystyle\prod\limits_{\varrho =1}^{i}\alpha _{s_{\varrho }}\text{, for }j\geq i>0
\end{array}%
\right\}   \label{pPolysValWRightIndZero}
\end{eqnarray}
and the polynomials $r_{l}=r_{l}(\alpha _{1}$, \ldots , $\alpha
_{l})$ in $\alpha _{1}$,\ldots , $\alpha _{l}$, for $l\geq 1$, as the expressions%
\begin{eqnarray}
r_{1}:=0\text{, }r_{l}(\alpha _{1}\text{, \ldots , }%
\alpha _{l}):=\dsum\limits_{\substack{ s_{1}+\cdots +s_{m}=m+l  \\ 1\leq
s_{\varrho }\leq l}}\prod\limits_{\varrho =1}^{m}\alpha _{s_{\varrho }}%
\text{, for }l>1 \label{rPolysVal}
\end{eqnarray}%
(see Appendix B for more details on these polynomials including recursive formulas for their calculation).

With these preliminaries we can now state the main results of this paper.  Proofs of these results are contained in the next section.

\begin{theorem}\label{MainTheorem}
\label{Main Results Theorem}Let $A\left( \varepsilon \right) $ be a
matrix-valued function having a range in $%
\mathbb{C}
^{n\times n}$ such that its matrix elements are analytic functions of $%
\varepsilon $ in a neighborhood of the origin.  Let $\lambda _{0}$ be an
eigenvalue of the unperturbed matrix $A\left( 0\right) $ and denote by $m$
its algebraic multiplicity.  Suppose that the generic condition%
\begin{equation}
\frac{\partial }{\partial \varepsilon }\det \left( \lambda I-A\left(
\varepsilon \right) \right) \big |_{\left( \varepsilon ,\lambda \right)
=\left( 0,\lambda _{0}\right) }\neq 0\text{,}
\label{MainResultsTheoremGenericCondition}
\end{equation}%
is true.  Then there is exactly one convergent Puiseux series for the $%
\lambda _{0}$-group and one for their corresponding eigenvectors whose branches are given by%
\begin{eqnarray}
\lambda _{h}\left( \varepsilon \right) &=&\lambda
_{0}+\sum\limits_{k=1}^{\infty }\alpha _{k}\left( \zeta ^{h}\varepsilon ^{%
\frac{1}{m}}\right) ^{k}  \label{MainResultsTheoremEigenvaluePuiseuxSeries}
\\
x_{h}\left( \varepsilon \right) &=&\beta_{0}+\sum\limits_{k=1}^{\infty
}\beta_{k}\left( \zeta ^{h}\varepsilon ^{\frac{1}{m}}\right) ^{k}
\label{MainResultsTheoremEigenvectorPuiseuxSeries}
\end{eqnarray}%
for $h=0, \ldots ,m-1$ and any fixed branch of $\varepsilon^{\frac{1}{m}}$, where $\zeta =e^{\frac{2\pi }{m}i}$ with %
\begin{equation*}
\alpha _{1}^{m}=\left( v_{m},A_{1}u_{1}\right) =-\frac{\frac{\partial }{%
\partial \varepsilon }\det \left( \lambda I-A\left( \varepsilon \right)
\right) |_{\left( \varepsilon ,\lambda \right) =\left( 0,\lambda _{0}\right)
}}{\left( \frac{\frac{\partial ^{m}}{\partial \lambda ^{m}}\det \left(
\lambda I-A\left( \varepsilon \right) \right) |_{\left( \varepsilon ,\lambda
\right) =\left( 0,\lambda _{0}\right) }}{m!}\right) }\neq 0
\end{equation*}%
$($Here $A_1$ denotes $\frac{dA}{d\varepsilon}(0)$ and the vectors $u_1\text{ and }v_m$ are defined in \eqref{BiorthogonalVectors} and \eqref{BiorthogonalVectors_v}$)$.
Furthermore, we can choose 
\begin{equation}
\alpha _{1}=\left( v_{m},A_{1}u_{1}\right) ^{1/m}\text{,}
\label{MainResultsTheoremFirstOrderTermChoice}
\end{equation}%
for any fixed $m$th root of $\left( v_{m},A_{1}u_{1}\right) $ and the
eigenvectors to satisfy the normalization conditions 
\begin{equation}
\left( v_{1},x_{h}\left( \varepsilon \right) \right) =1\text{,}~h=0,...,m-1%
\text{.}  \label{MainResultsTheoremNormalizationCondition}
\end{equation}%
Consequently, under these conditions $\alpha _{1}$, $\alpha _{2}$,\ldots\
and $\beta_{0}$, $\beta_{1}$, \ldots\ are uniquely determined and are given by the
recursive formulas%
\begin{eqnarray}
\alpha _{1} &=&\left( v_{m},A_{1}u_{1}\right) ^{1/m}=\left( -\frac{\frac{%
\partial }{\partial \varepsilon }\det \left( \lambda I-A\left( \varepsilon
\right) \right) |_{\left( \varepsilon ,\lambda \right) =\left( 0,\lambda
_{0}\right) }}{\left( \frac{\frac{\partial ^{m}}{\partial \lambda ^{m}}\det
\left( \lambda I-A\left( \varepsilon \right) \right) |_{\left( \varepsilon
,\lambda \right) =\left( 0,\lambda _{0}\right) }}{m!}\right) }\right) ^{1/m}
\label{MainResultsTheoremRecursiveFormulas1} \\
\alpha _{s} &=&\frac{-r_{s-1}+\sum\limits_{i=0}^{\min\{s,m\}-1}\sum\limits_{j=i}^{s-1}p_{j,i}%
\left( v_{m-i},\sum\limits_{k=1}^{\left\lfloor \frac{m+s-1-j}{m}%
\right\rfloor }A_{k}\beta_{m+s-1-j-km}\right) }{m\alpha _{1}^{m-1}}  \label{MainResultsTheoremRecursiveFormulas2} \\
\beta_{s} &=&\left\{ 
\begin{array}{c}
\sum\limits_{i=0}^{s}p_{s,i}u_{i+1}\text{, if }0\leq s\leq m-1 \\ 
\sum\limits_{i=0}^{m-1}p_{s,i}u_{i+1}-\sum\limits_{j=0}^{s-m}\sum%
\limits_{k=0}^{j}\sum\limits_{l=1}^{\left\lfloor \frac{s-j}{m}\right\rfloor
}p_{j,k}\Lambda ^{k+1}A_{l}\beta_{s-j-lm}\text{, if }s\geq m%
\end{array}%
\right.  \label{MainResultsTheoremRecursiveFormulas3}
\end{eqnarray}%
where $u_i$ and $v_i$ are the vectors defined in $\eqref{BiorthogonalVectors}$ and $\eqref{BiorthogonalVectors_v}$, $p_{j,i}$ and $r_{l}$ are the polynomials defined in \eqref{pPolysValWRightIndZero} and \eqref{rPolysVal}, $\left\lfloor {}\right\rfloor $ denotes the floor function, $A_k$ denotes the matrix $\frac{1}{k!}\frac{d^k A}{d\varepsilon^k}(0)$, and $\Lambda$ is the matrix defined in \eqref{PartialInverseOfTheJordanNormalFormOfANot}.

\begin{corollary}\label{MurdockClarkMyResponseCorollary}
\label{Calculation kth order by matrices corollary}The calculation of the $k$%
th order terms, $\alpha _{k}~$and $\beta_{k}$, requires only the matrices $%
A_{0} $, \ldots , $A_{\left\lfloor \frac{m+k-1}{m}\right\rfloor }$.
\end{corollary}

\begin{corollary}\label{2ndOrderCoeffCorollary}
\label{Second order coefficients corollary}The coefficients of those Puiseux
series up to second order are given by%
\begin{eqnarray*}
\alpha _{1} &=&\left( -\frac{\frac{\partial f}{\partial \varepsilon }\left(
0,\lambda _{0}\right) }{\frac{1}{m!}\frac{\partial ^{m}f}{\partial \lambda
^{m}}\left( 0,\lambda _{0}\right) }\right) ^{1/m}=\left(
v_{m},A_{1}u_{1}\right) ^{1/m}\text{,} \\
\alpha _{2} &=&\left\{ 
\begin{array}{c}
\frac{-\left( \alpha _{1}^{m+1}\frac{1}{\left( m+1\right) !}\frac{\partial
^{m+1}f}{\partial \lambda ^{m+1}}\left( 0,\lambda _{0}\right) +\alpha _{1}%
\frac{\partial ^{2}f}{\partial \lambda \partial \varepsilon }\left(
0,\lambda _{0}\right) +\frac{1}{2}\frac{\partial ^{2}f}{\partial \varepsilon
^{2}}\left( 0,\lambda _{0}\right) \right) }{m\alpha _{1}^{m-1}\left( \frac{1%
}{m!}\frac{\partial ^{m}f}{\partial \lambda ^{m}}\left( 0,\lambda
_{0}\right) \right) }\text{, if }m=1 \\ 
\frac{-\left( \alpha _{1}^{m+1}\frac{1}{\left( m+1\right) !}\frac{\partial
^{m+1}f}{\partial \lambda ^{m+1}}\left( 0,\lambda _{0}\right) +\alpha _{1}%
\frac{\partial ^{2}f}{\partial \lambda \partial \varepsilon }\left(
0,\lambda _{0}\right) \right) }{m\alpha _{1}^{m-1}\left( \frac{1}{m!}\frac{%
\partial ^{m}f}{\partial \lambda ^{m}}\left( 0,\lambda _{0}\right) \right) }%
\text{, if }m>1%
\end{array}%
\right. \\
&=&\left\{ 
\begin{array}{c}
\left( v_{1},\left( A_{2}-A_{1}\Lambda A_{1}\right) u_{1}\right) \text{, if }%
m=1 \\ 
\frac{\left( v_{m-1},A_{1}u_{1}\right) +\left( v_{m},A_{1}u_{2}\right) }{%
m\alpha _{1}^{m-2}}\text{, if }m>1%
\end{array}%
\right. \text{,} \\
\beta_{0} &=&u_{1}\text{,}\\
\beta_{1} &=&\left\{ 
\begin{array}{c}
-\Lambda A_{1}u_{1}\text{, if }m=1 \\ 
\alpha _{1}u_{2}\text{, if }m>1%
\end{array}%
\right. \text{,} \\
\beta_{2} &=&\left\{ 
\begin{array}{c}
\left( -\Lambda A_{2}+\left( \Lambda A_{1}\right) ^{2}-\alpha _{1}\Lambda
^{2}A_{1}\right) u_{1}\text{, if }m=1 \\ 
-\Lambda A_{1}u_{1}+\alpha _{2}u_{2}\text{, if }m=2 \\ 
\alpha _{2}u_{2}+\alpha _{1}^{2}u_{3}\text{, if }m>2%
\end{array}%
\right. \text{,}
\end{eqnarray*}%
where $f\left( \varepsilon ,\lambda \right) :=\det \left( \lambda I-A\left(
\varepsilon \right) \right) $.
\end{corollary}
\end{theorem}
\begin{remark}\label{RemarkSteps}
Suppose we want to calculate the terms $\alpha_{k+1},\beta_{k+1}$, where $k\geq 2$, using the explicit recursive formulas given in the theorem.  We may assume we already known or have calculated
\begin{eqnarray} \label{Step1}
A_{0}, \ldots , A_{\left\lfloor \frac{m+k}{m}\right\rfloor} \text{, }\{r_j\}_{j=1}^{k-1}\text{, } \{\alpha_j\}_{j=1}^{k}\text{, } \{\beta_j\}_{j=0}^{k} \text{, }\{\{p_{j,i}\}_{j=i}^{k}\}_{i=0}^{k}.
\end{eqnarray}
We need these to calculate $\alpha_{k+1},\beta_{k+1}$ and the steps to do this are indicated by the following arrow diagram:
\begin{eqnarray}\label{TheSteps}
\eqref{Step1}\overset{\eqref{rPolysRecursFormula}}{\rightarrow} r_{k} \overset{\eqref{MainResultsTheoremRecursiveFormulas2}}{\rightarrow}\alpha_{k+1}\overset{\eqref{pPolysRecursFormula}}{\rightarrow}\{p_{k+1,i}\}_{i=0}^{k+1}\overset{\eqref{MainResultsTheoremRecursiveFormulas3}}{\rightarrow}\beta_{k+1}.
\end{eqnarray}
After we have followed these steps we not only will have calculated $\alpha_{k+1},\beta_{k+1}$ but we will now know
\begin{eqnarray}\label{Back2Step1}
A_{0}, \ldots , A_{\left\lfloor \frac{m+k+1}{m}\right\rfloor} \text{, }\{r_j\}_{j=1}^{k}\text{, } \{\alpha_j\}_{j=1}^{k+1}\text{, } \{\beta_j\}_{j=0}^{k+1} \text{, }\{\{p_{j,i}\}_{j=i}^{k+1}\}_{i=0}^{k+1}
\end{eqnarray}
as well.  But these are the terms in \eqref{Step1} for $k+1$ and so we may repeat the steps indicated above to calculate $\alpha_{k+2},\beta_{k+2}$.

It is in this way we see how all the higher order terms can be calculated using the results of this paper.
\end{remark}

\subsection*{Example}
In order to illustrate these steps we give the following example which recursively calculates the third order terms for $m\geq 3$.

The goal is to determine $\alpha_3,\beta_3$. To do this we follow the steps indicated in the above remark with $k=2$.  The first step is to collect the terms in \eqref{Step1}.  Assuming $A_0$, $A_1$ are known then by \eqref{pPolysValWRightIndZero}, \eqref{rPolysVal}, Corollary \ref{2ndOrderCoeffCorollary}, and Proposition \ref{Appendix B Proposition} we have
\begin{equation*}
\begin{array}{c}
A_0 \text{, } A_1 \text{, } r_1 = 0 \text{, } \alpha _{1} =\left(
v_{m},A_{1}u_{1}\right) ^{1/m}\text{, }\alpha _{2}=\frac{\left( v_{m-1},A_{1}u_{1}\right) +\left( v_{m},A_{1}u_{2}\right) }{%
m\alpha _{1}^{m-2}}, \\ 
\beta_{0} = u_{1}\text{, }\beta_{1}=\alpha _{1}u_{2} \text{, }\beta_{2}=\alpha _{2}u_{2}+\alpha _{1}^{2}u_{3}, \\ 
p_{0,0}=1 \text{, } p_{1,0}=0 \text{, } p_{1,1}=\alpha_1 \text{, } p_{2,0}=0 \text{, } p_{2,1}=\alpha_2 \text{, } p_{2,2}=\alpha_1^2.%
\end{array}%
\end{equation*}

The next step is to determine $r_2$ using the recursive formula for the $r_l$'s given in \eqref{rPolysRecursFormula}.  We find that
\begin{eqnarray*}
r_2&=&\frac{1}{2\alpha_1}\tsum_{j=1}^{1}[(3-j)m-(m+j)]\alpha_{3-j}r_j+ \frac{m}{2}\alpha_1^{m-2}\tsum_{j=1}^{1}[(3-j)m-(m+j)]\alpha_{3-j}\alpha_{j+1} \notag\\
&=&\frac{m(m-1)}{2}\alpha_1^{m-2}\alpha_2^2.
\end{eqnarray*}

Now, since $r_2$ is determined, we can use the recursive formula in \eqref{MainResultsTheoremRecursiveFormulas2} for the $\alpha_s$'s to calculate $\alpha_3$.  In doing so we find that
\begin{eqnarray*}
\alpha _{3} &=&\frac{-r_{2}+\sum\limits_{i=0}^{\min
\{3,m\}-1}\sum\limits_{j=i}^{2}p_{j,i}\left(
v_{m-i},\sum\limits_{k=1}^{\left\lfloor \frac{m+2-j}{m}\right\rfloor
}A_{k}\beta _{m+2-j-km}\right) }{m\alpha _{1}^{m-1}} \\
&=&\frac{-r_{2}+p_{2,1}\left( v_{m-1},A_{1}\beta _{0}\right) +p_{0,0}\left(
v_{m},A_{1}\beta _{2}\right) }{m\alpha _{1}^{m-1}}+ \\
&&\frac{p_{2,2}\left( v_{m-2},A_{1}\beta _{0}\right) +p_{1,1}\left(
v_{m-1},A_{1}\beta _{1}\right) }{m\alpha _{1}^{m-1}} \\
&=&\frac{-\frac{m(m-1)}{2}\alpha _{1}^{m-2}\alpha _{2}^{2}+\alpha _{2}\left(
v_{m-1},A_{1}u_{1}\right) +\left( v_{m},A_{1}(\alpha _{2}u_{2}+\alpha
_{1}^{2}u_{3})\right) }{m\alpha _{1}^{m-1}}+ \\
&&\frac{\alpha _{1}^{2}\left( v_{m-2},A_{1}u_{1}\right) +\alpha _{1}\left(
v_{m-1},A_{1}\alpha _{1}u_{2}\right) }{m\alpha _{1}^{m-1}} \\
&=&\left( \frac{3-m}{2}\right) \alpha _{1}^{-1}\alpha _{2}^{2}+\frac{\left(
v_{m-2},A_{1}u_{1}\right) +\left( v_{m-1},A_{1}u_{2}\right) +\left(
v_{m},A_{1}u_{3}\right) }{m\alpha _{1}^{m-3}}.
\end{eqnarray*}

Next, since $\alpha_3$ is determined, we can use \eqref{pPolysRecursFormula} to calculate $\{p_{3,i}\}_{i=0}^{3}$.  In this case though it suffices to use Proposition \ref{Appendix B Proposition} and in doing so we find that
\begin{eqnarray*}
p_{3,0}=0\text{, }p_{3,1}=\alpha_3\text{, } p_{3,2}=2\alpha_1\alpha_2\text{, } p_{3,3}=\alpha_1^3.
\end{eqnarray*}

Finally, we can compute $\beta_3$ using the recursive formula in \eqref{MainResultsTheoremRecursiveFormulas3} for the $\beta_s$'s.  In doing so we find that
\begin{eqnarray*}
\beta_{3} &=&\left\{ 
\begin{array}{c}
\sum\limits_{i=0}^{3}p_{3,i}u_{i+1}\text{, if }m> 3 \\ 
\sum\limits_{i=0}^{m-1}p_{3,i}u_{i+1}-\sum\limits_{j=0}^{3-m}\sum%
\limits_{k=0}^{j}\sum\limits_{l=1}^{\left\lfloor \frac{3-j}{m}\right\rfloor
}p_{j,k}\Lambda ^{k+1}A_{l}\beta_{3-j-lm}\text{, if }m=3
\end{array}
\right.\\
&=&\left\{ 
\begin{array}{c}
p_{3,1}u_2+p_{3,2}u_3+p_{3,3}u_4\text{, if }m> 3 \\ 
\sum\limits_{i=0}^{2}p_{3,i}u_{i+1}-\Lambda A_{1}\beta_0\text{, if }m=3
\end{array}
\right.\\
&=&\left\{ 
\begin{array}{c}
\alpha_3u_2+2\alpha_1\alpha_2u_3+\alpha_1^3u_4\text{, if }m> 3 \\ 
\alpha_3u_2+2\alpha_1\alpha_2u_3-\Lambda A_{1}u_1\text{, if }m=3.
\end{array}
\right.
\end{eqnarray*}
This completes the calculation of the third order terms, $\alpha_3,\beta_3$, when $m\geq 3$.

\section{Proofs}

This section contains the proofs of the results of this paper.  We begin by
proving Theorem \ref{Generic Condition Theorem} of \S 2 on conditions
equivalent to the generic condition.  We next follow this up with the proof
of the main result of this paper Theorem \ref{Main Results Theorem}.  We
finish by proving the Corollaries \ref{Calculation kth order by matrices
corollary} and \ref{Second order coefficients corollary}.

\subsection{Proof of Theorem \protect\ref{Generic Condition Theorem}}
To prove this theorem we will prove the following chain of statements (i)$%
\Rightarrow $(ii)$\Rightarrow $(iii)$\Rightarrow $(iv)$\Rightarrow $(i).

We begin by proving (i)$\Rightarrow$(ii).  Define $%
f\left( \varepsilon ,\lambda \right) :=\det \left( \lambda I-A\left(
\varepsilon \right) \right) $ and suppose (i) is true.  Then $f$ is an analytic function of $\left(
\varepsilon ,\lambda \right) $ near $\left( 0,\lambda _{0}\right) $ since
the matrix elements of $A\left( \varepsilon \right) $ are analytic functions
of $\varepsilon $ in a neighborhood of the origin and the determinant of a
matrix is a polynomial in its matrix elements.  Also we have $f\left(
0,\lambda _{0}\right) =0$ and $\frac{\partial f}{\partial \varepsilon }%
\left( 0,\lambda _{0}\right) \neq 0$.  Hence by the holomorphic implicit
function theorem \cite[\S 1.4 Theorem 1.4.11]{Krantz1992} there exists a unique
solution, $\varepsilon \left( \lambda \right) $, in a neighborhood of $%
\lambda =\lambda _{0}$ with $\varepsilon \left( \lambda _{0}\right) =0$ to
the equation $f\left( \varepsilon ,\lambda \right) =0$, which is analytic at 
$\lambda =\lambda _{0}$.  We now show that $\varepsilon \left( \lambda
\right) $ has a zero there of order $m$ at $\lambda =\lambda _{0}$.  First, the properties of $\varepsilon(\lambda)$ imply there exists $\varepsilon _{q}\neq 0$ and $q\in\mathbb{N}$ such that $\varepsilon \left( \lambda \right) =\varepsilon _{q}\left(\lambda -\lambda _{0}\right) ^{q}+O\left( \left( \lambda -\lambda_{0}\right) ^{q+1}\right) $, for $|\lambda-\lambda_0|<<1$.  Next, by hypothesis $\lambda _{0}$ is an eigenvalue of $A\left( 0\right) $ of algebraic multiplicity $m$ hence $\partial ^{i}f\backslash \partial \lambda ^{i}\left( 0,\lambda _{0}\right) $ $=0$ for $0\leq i\leq m-1$ but $\partial ^{m}f\backslash \partial \lambda^{m}\left( 0,\lambda _{0}\right) $ $\not=0$.  Combining this with the fact that $f\left( 0,\lambda _{0}\right) =0$ and $\frac{\partial f}{\partial \varepsilon }\left( 0,\lambda _{0}\right) \neq 0$ we have
\begin{equation}
f\left( \varepsilon ,\lambda \right) =a_{10}\varepsilon +a_{0m}\left(
\lambda -\lambda _{0}\right) ^{m}+\sum_{\substack{ i+j\geq 2\text{, }i,j\in 
\mathbb{N}
\\ \left( i,j\right) \not\in \{\left( 0,j\right) \text{:}j\leq m\}}}a_{ij}\varepsilon ^{i}\left( \lambda -\lambda _{0}\right) ^{j}
\label{SeriesRepresentationForf}
\end{equation}%
for $|\varepsilon|+|\lambda-\lambda_0|<<1$, where $a_{10}=\frac{\partial f}{\partial \varepsilon }\left( 0,\lambda
_{0}\right) \neq 0$ and $a_{0m}=\frac{1}{m!}\frac{\partial ^{m}f}{\partial
\lambda ^{m}}\left( 0,\lambda _{0}\right) \neq 0$.  Then
using the expansions of $f\left( \varepsilon ,\lambda \right) $ and $\varepsilon \left( \lambda \right) $ together with the identity $f\left(
\varepsilon \left( \lambda \right) ,\lambda \right) =0$ for $|\lambda
-\lambda _{0}|<<1$, we find that 
$q=m$ and 
\begin{equation}
\varepsilon _{m}=-\frac{a_{0m}}{a_{10}}=-\frac{\frac{1}{m!}\frac{\partial
^{m}\det \left( \lambda I-A\left( \varepsilon \right) \right) }{\partial
\lambda ^{m}}\Big |_{\left( \lambda ,\varepsilon \right) =\left( \lambda
_{0},0\right) }}{\frac{\partial }{\partial \varepsilon }\det \left( \lambda
I-A\left( \varepsilon \right) \right) \big |_{\left( \lambda ,\varepsilon
\right) =\left( \lambda _{0},0\right) }}\text{.}
\label{RelatedToFirstOrderTermOfPerturbedEigenvalues}
\end{equation}%
Therefore we conclude that $\varepsilon \left( \lambda \right) $ has a zero
of order $m$ at $\lambda =\lambda _{0}$, which proves (ii).

Next, we prove (ii)$\Rightarrow$(iii).  Suppose (ii) is true.  The
first part of proving (iii) involves inverting $\varepsilon \left( \lambda
\right) $ near $\varepsilon =0$ and $\lambda =\lambda _{0}$.  To do this we
expand $\varepsilon \left( \lambda \right) $ in a power series about $%
\lambda =\lambda _{0}$ and find that $\varepsilon \left( \lambda \right)
=g(\lambda )^{m}$ where%
\begin{equation*}
g(\lambda )=\left( \lambda -\lambda _{0}\right) \left( \varepsilon
_{m}+\sum_{k=m+1}^{\infty }\varepsilon _{k}\left( \lambda -\lambda
_{0}\right) ^{k-m}\right) ^{1/m}
\end{equation*}%
and we are taking any fixed branch of the $m$th root that is analytic at $%
\varepsilon _{m}$.  Notice that, for $\lambda $ in a small enough
neighborhood of $\lambda _{0}$, $g$ is an analytic function, $g\left(
\lambda _{0}\right) =0$, and $\frac{dg}{d\lambda }(\lambda _{0})=\varepsilon
_{m}^{1/m}\neq 0$.  This implies, by the inverse function theorem for
analytic functions, that for $\lambda $ in a small enough neighborhood of $%
\lambda _{0}$ the analytic function $g\left( \lambda \right) $ has an
analytic inverse $g^{-1}\left( \varepsilon \right) $ in a neighborhood of $%
\varepsilon =0$ with $g^{-1}\left( 0\right) =\lambda _{0}$.  Define a
multivalued function $\lambda \left( \varepsilon \right) $, for sufficiently
small $\varepsilon $, by $\lambda \left( \varepsilon \right) :=g^{-1}\left(
\varepsilon ^{\frac{1}{m}}\right) $ where by $\varepsilon ^{\frac{1}{m}}$ we
mean all branches of the $m$th root of $\varepsilon $.  We know that $%
g^{-1} $ is analytic at $\varepsilon =0$ so that for sufficiently small $%
\varepsilon $ the multivalued function $\lambda \left( \varepsilon \right) $
is a Puiseux series.  And since $\frac{dg^{-1}}{d\varepsilon }(0)=\left[ 
\frac{dg}{d\lambda }(\lambda _{0})\right] ^{-1}\neq 0$ we have an expansion 
\begin{equation*}
\lambda \left( \varepsilon \right) =g^{-1}\left( \varepsilon ^{\frac{1}{m}%
}\right) =\lambda _{0}+\alpha _{1}\varepsilon ^{\frac{1}{m}%
}+\sum_{k=2}^{\infty }\alpha _{k}\left( \varepsilon ^{\frac{1}{m}}\right)
^{k}\text{.}
\end{equation*}%
Now suppose for fixed $\lambda $ sufficiently near $\lambda _{0}$ and
for sufficiently small $\varepsilon $ we have $\det \left( \lambda I-A\left(
\varepsilon \right) \right) =0$.  We want to show this implies $%
\lambda =\lambda \left( \varepsilon \right) $ for one of the branches of the 
$m$th root.  We know by hypothesis we must have $\varepsilon
=\varepsilon \left( \lambda \right) $.  But as we know this implies that $%
\varepsilon =\varepsilon \left( \lambda \right) =g(\lambda )^{m}$ hence for
some branch of the $m$th root, $b_{m}(\cdotp)$, we have $b_{m}(\varepsilon
)=b_{m}(g(\lambda )^{m})=g(\lambda )$.  But $\lambda $ is near enough to $%
\lambda _{0}$ and $\varepsilon $ is sufficiently small that we may apply the 
$g^{-1}$ to both sides yielding $\lambda =g^{-1}\left( g(\lambda )\right)
=g^{-1}\left( b_{m}(\varepsilon )\right) =\lambda \left( \varepsilon \right) 
$, as desired.  Furthermore, all the $m$ branches $\lambda _{h}\left(
\varepsilon \right) $, $h=0, \ldots ,m-1$ of $\lambda \left( \varepsilon \right) $
are given by taking all branches of the $m$th root of $\varepsilon $ so that%
\begin{equation*}
\lambda _{h}\left( \varepsilon \right) =\lambda _{0}+\alpha_{1}\zeta
^{h}\varepsilon ^{\frac{1}{m}}+\sum_{k=2}^{\infty }\alpha_{k}\left( \zeta
^{h}\varepsilon ^{\frac{1}{m}}\right) ^{k}
\end{equation*}%
for any fixed branch of $\varepsilon^{\frac{1}{m}}$, where $\zeta =e^{\frac{2\pi }{m}i}$ and 
\begin{equation}
\alpha_{1}=\frac{dg^{-1}}{d\varepsilon }(0)=\left[ \frac{dg}{d\lambda }%
(\lambda _{0})\right] ^{-1}=\varepsilon _{m}^{-1/m}\neq 0\text{,}
\label{FirstOrderTermForPerturbedEigenvalues}
\end{equation}
which proves (iii).

Next, we prove (iii)$\Rightarrow $(iv).  Suppose (iii) is true.  Define the function $y\left( \varepsilon \right) :=\lambda _{0}\left(
\varepsilon ^{m}\right) $.  Then $y$ is analytic at $\varepsilon =0$ and $%
\frac{dy}{d\varepsilon }\left( 0\right) =\lambda _{1}\neq 0$.  Also we have
for $\varepsilon $ sufficiently small $\det \left( y\left( \varepsilon
\right) I-A\left( \varepsilon ^{m}\right) \right) =0$.  Consider the
inverse of $y\left( \varepsilon \right) $, $y^{-1}\left( \lambda \right) $.  It satisfies $0=\det \left( y\left( y^{-1}\left( \lambda \right) \right)
I-A\left( \left[ y^{-1}\left( \lambda \right) \right] ^{m}\right) \right)
=\det \left( \lambda I-A\left( \left[ y^{-1}\left( \lambda \right) \right]
^{m}\right) \right) $ with $y^{-1}\left( \lambda _{0}\right) =0$, $\frac{%
dy^{-1}}{d\lambda }\left( \lambda _{0}\right) =\alpha_{1}^{-1}$.  Define $%
g\left( \lambda \right) :=\left[ y^{-1}\left( \lambda \right) \right] ^{m}$.  Then $g$ has a zero of order $m$ at $\lambda _{0}$ and $\det \left(
\lambda I-A\left( g\left( \lambda \right) \right) \right) =0$ for $\lambda $
in a neighborhood of $\lambda _{0}$.

Now we consider the analytic matrix $A\left( g\left( \lambda \right) \right)
-\lambda I$ in a neighborhood of $\lambda =\lambda _{0}$ with the constant
eigenvalue $0$.  Because $0$ is an analytic eigenvalue of it then there
exists an analytic eigenvector, $x\left( \lambda \right) $, of $A\left(
g\left( \lambda \right) \right) -\lambda I$ corresponding to the eigenvalue $%
0$ in a neighborhood of $\lambda _{0}$ such that $x\left( \lambda
_{0}\right) \neq 0$. Hence for $\lambda$ near $\lambda_0$ we have
\begin{eqnarray*}
0 &=&\left( A\left( g\left( \lambda \right) \right) -\lambda I\right)
x\left( \lambda \right) \\
&=&\left( A\left( \alpha_{1}^{-m}\left( \lambda -\lambda _{0}\right)
^{m}+O\left( \left( \lambda -\lambda _{0}\right) ^{m+1}\right) \right)
-\left( \lambda -\lambda _{0}\right) I-\lambda _{0}I\right) x\left( \lambda
\right) \\
&=&\left( A\left( 0\right) -\lambda _{0}I\right) x\left( \lambda _{0}\right)
+\left( \left( A\left( 0\right) -\lambda _{0}I\right) \frac{dx}{d\lambda }%
\left( \lambda _{0}\right) -x\left( \lambda _{0}\right) \right) \left(
\lambda -\lambda _{0}\right) +\cdots \\
&&+\left( \left( A\left( 0\right) -\lambda _{0}I\right) \frac{d^{m-1}x}{%
d\lambda ^{m-1}}\left( \lambda _{0}\right) -\frac{d^{m-2}x}{d\lambda ^{m-2}}%
\left( \lambda _{0}\right) \right) \left( \lambda -\lambda _{0}\right) ^{m-1}
\\
&&+\left( \left( A\left( 0\right) -\lambda _{0}I\right) \frac{d^{m}x}{%
d\lambda ^{m}}\left( \lambda _{0}\right) -\frac{d^{m-1}x}{d\lambda ^{m-1}}%
\left( \lambda _{0}\right) +\alpha_{1}^{-m}A^{\prime}(0) x\left( \lambda _{0}\right) \right) \left( \lambda -\lambda
_{0}\right) ^{m} \\
&&+O\left( \left( \lambda -\lambda _{0}\right) ^{m+1}\right) \text{.}
\end{eqnarray*}%
This implies that%
\begin{eqnarray}
&&\left( A\left( 0\right) -\lambda _{0}I\right) x\left( \lambda _{0}\right)= 0,
\left( A\left( 0\right) -\lambda _{0}I\right) \frac{d^jx}{d\lambda^j }\left(
\lambda _{0}\right) = \frac{d^{j-1}x}{d\lambda ^{j-1}}\left(
\lambda _{0}\right) \text{, for }j=1, \ldots, m-1, \notag\\
&&\left( A\left( 0\right) -\lambda _{0}I\right) \frac{d^{m}x}{d\lambda ^{m}}%
\left( \lambda _{0}\right) =\frac{d^{m-1}x}{d\lambda ^{m-1}}\left( \lambda
_{0}\right) -\alpha_{1}^{-m}A^{\prime}(0)
x\left( \lambda _{0}\right) \text{.} \label{MPlus1OrderEquality}
\end{eqnarray}

The first $m$ equations imply that $x\left( \lambda _{0}\right) $, $%
\frac{dx}{d\lambda }\left( \lambda _{0}\right) $, $\ldots $, $\frac{d^{m-1}x%
}{d\lambda ^{m-1}}\left( \lambda _{0}\right) $ is a Jordan chain of length $%
m $ generated by $\frac{d^{m-1}x}{d\lambda ^{m-1}}\left( \lambda _{0}\right) 
$.  Since the algebraic multiplicity of $\lambda _{0}$ for $A\left(
0\right) $ is $m$ this implies that the there is a single $m\times m$ Jordan
block corresponding to the eigenvalue $\lambda _{0}$ where we can take $%
x\left( \lambda _{0}\right) ,\frac{dx}{d\lambda }\left( \lambda _{0}\right)
,...,\frac{d^{m-1}x}{d\lambda ^{m-1}}\left( \lambda _{0}\right) $ as a
Jordan basis. It follows from basic properties of Jordan chains that there exists an eigenvector $v$ of $A(0)^*$ corresponding to the eigenvalue $\lambda_0^*$ such that $%
\left( v,\frac{d^{m-1}x}{d\lambda ^{m-1}}\left( \lambda _{0}\right)
\right) = 1$. Hence%
\begin{eqnarray*}
0 &=& \left((A(0)-\lambda_0I)^*v,\frac{d^{m}x}{d\lambda ^{m}}\left( \lambda _{0}\right)\right) \overset{\eqref{MPlus1OrderEquality}}{=} 1-\alpha_{1}^{-m}\left( v,A^{\prime}(0)
x\left( \lambda _{0}\right) \right)
\end{eqnarray*}%
implying that $\left( v,\frac{dA}{d\varepsilon }\left( 0\right) x\left(
\lambda _{0}\right) \right) =\alpha_{1}^{m}\neq 0$.  Therefore we have
shown that the Jordan normal form of $A\left( 0\right) $ corresponding to
the eigenvalue $\lambda _{0}$ consists of a single $m\times m$ Jordan block
and there exists an eigenvector $u$ of $A\left( 0\right) $ corresponding to
the eigenvalue $\lambda _{0}$ and an eigenvector $v$ of $A\left( 0\right)
^{\ast }$ corresponding to the eigenvalue $\lambda _{0}^{\ast }$ such that $%
\left( v,A^{\prime}(0)u\right) \neq 0$.  This proves (iv).

Finally, we show (iv)$\Rightarrow $(i).  Suppose (iv) is true.  We
begin by noting that since 
\begin{eqnarray*}
\det \left( \lambda _{0}I-A\left( \varepsilon \right) \right) &=&\left( -1\right) ^{n}\det \left( \left( A\left( 0\right) -\lambda
_{0}I\right) +A^\prime(0) \varepsilon \right) +o\left(
\varepsilon \right)
\end{eqnarray*}%
it suffices to show that%
\begin{equation}\label{ProveThisIsNonZero}
S_{n-1}:=\frac{d}{d\varepsilon }\det \left( \left( A\left( 0\right) -\lambda
_{0}I\right) +A^\prime(0) \varepsilon \right) \big |%
_{\varepsilon =0}\neq 0\text{.}
\end{equation}

We will use the result from \cite[Theorem 2.16]{IpsenRehman2008} to prove \eqref{ProveThisIsNonZero}.  Let $A\left( 0\right) -\lambda _{0}I=Y\Sigma X^{\ast }$ be a singular-value
decomposition of the matrix $A\left( 0\right) -\lambda _{0}I$ where $X$, $Y$
are unitary matrices and $\Sigma =\diag(\sigma _{1},\ldots ,\sigma
_{n-1},\sigma _{n})$ with $\sigma _{1}\geq \ldots \geq \sigma _{n-1}\geq
\sigma _{n}\geq 0$ (see \cite[\S 5.7, Theorem 2]{LancasterTismenetsky1985}).  Now
since the Jordan normal form of $A\left( 0\right) $ corresponding to the
eigenvalue $\lambda _{0}$ consists of a single Jordan block this implies
that rank of $A\left( 0\right) -\lambda _{0}I$ is $n-1$.  This implies that 
$\sigma _{1}\geq \ldots \geq \sigma _{n-1}>\sigma _{n}=0$, $u=Xe_{n}$ is an
eigenvalue of $A\left( 0\right) $ corresponding to the eigenvalue $\lambda
_{0}$, $v=Ye_{n}$ is an eigenvalue of $A\left( 0\right) $ corresponding
to the eigenvalue $\lambda _{0}^{\ast }$, and there exist nonzero constants $c_{1}, c_{2}$ such that $u=c_{1}u_{0}$ and $v=c_{2}v_{0}$.

Now using the result of \cite[Theorem 2.16]{IpsenRehman2008} for \eqref{ProveThisIsNonZero} we find that%
\begin{equation*}
S_{n-1}=\det \left( YX^{\ast }\right) \sum_{1\leq i_{1}<\cdots <i_{n-1}\leq
n}\sigma _{i_{1}}\cdots \sigma _{i_{n-1}}\det \left( \left( Y^{\ast
}A^\prime(0) X\right) _{i_{1}\ldots i_{n-1}}\right) \text{,}
\end{equation*}%
where $\left( Y^{\ast }A^\prime(0) X\right) _{i_{1}\ldots
i_{n-1}}$ is the matrix obtained from $Y^{\ast }A^\prime(0) X$ by removing rows and columns $i_{1}\ldots i_{n-1}$. But since $%
\sigma _{n}=0$ and
\begin{eqnarray*}
\left( Y^{\ast }A^\prime(0) X\right) _{1\ldots (n-1)}=e_{n}^{\ast }Y^{\ast }A^\prime(0) Xe_{n}=\left( v,A^\prime(0) u\right)=c_{2}^{\ast }c_{1}\left( v_{0},A^\prime(0) u_{0}\right)
\neq 0
\end{eqnarray*}%
then $S_{n-1}=\det \left( YX^{\ast }\right) \tprod\limits_{j=1}^{n-1}\sigma
_{j}c_{2}^{\ast }c_{1}\left(v_{0},A^\prime(0)
u_{0}\right) \neq 0$.  This completes the proof.
\qquad \endproof

\subsection{Proof of Theorem \protect\ref{Main Results Theorem}}

We begin by noting that our hypotheses imply that statements (ii), (iii),
and (iv) of Theorem \ref{Generic Condition Theorem} are true.  In
particular, statement (iii) implies that there is exactly one convergent
Puiseux series for the $\lambda _{0}$-group whose branches are given by 
\begin{equation*}
\lambda _{h}\left( \varepsilon \right) =\lambda _{0}+\alpha_{1}\zeta
^{h}\varepsilon ^{\frac{1}{m}}+\sum\limits_{k=2}^{\infty }\alpha_{k}\left(
\zeta ^{h}\varepsilon ^{\frac{1}{m}}\right) ^{k}\text{,}
\end{equation*}%
for $h=0,\ldots ,m-1$ and any fixed branch of $\varepsilon^{\frac{1}{m}}$, where $\zeta =e^{\frac{2\pi }{m}i}$ and $\alpha_{1}\neq 0$.  Then by well
known results \cite[\S 6.1.7, Theorem 2]{Baumgartel1985}, \cite[\S II.1.8]{Kato1995} there exists a
convergent Puiseux series for the corresponding eigenvectors whose branches are given by 
\begin{equation*}
x_{h}\left( \varepsilon \right) =\beta_{0}+\sum\limits_{k=1}^{\infty
}\beta_{k}\left( \zeta ^{h}\varepsilon ^{\frac{1}{m}}\right) ^{k}\text{,}
\end{equation*}%
for $h=0,\ldots ,m-1$, where $\beta_{0}$ is an eigenvector of $A_{0}=A\left( 0\right) $ corresponding to the eigenvalue $\lambda _{0}$.  Now if we examine the proof of (ii)$\Rightarrow $(iii) in Theorem \ref%
{Generic Condition Theorem} we see by equation (\ref%
{FirstOrderTermForPerturbedEigenvalues}) that $\alpha_{1}^{m}=\varepsilon
_{m}^{-1}$, where $\varepsilon _{m}$ is given in equation (\ref%
{RelatedToFirstOrderTermOfPerturbedEigenvalues}) in the proof of (i)$%
\Rightarrow $(iii) for Theorem \ref{Generic Condition Theorem}.  Thus we
can conclude that%
\begin{equation}
\label{ValOflambda1^m}
\alpha_{1}^{m}=-\frac{\frac{\partial }{\partial \varepsilon }\det \left(
\lambda I-A\left( \varepsilon \right) \right) |_{\left( \varepsilon ,\lambda
\right) =\left( 0,\lambda _{0}\right) }}{\left( \frac{\frac{\partial ^{m}}{%
\partial \lambda ^{m}}\det \left( \lambda I-A\left( \varepsilon \right)
\right) |_{\left( \varepsilon ,\lambda \right) =\left( 0,\lambda _{0}\right)
}}{m!}\right) }\neq 0\text{.}
\end{equation}

Choose any $m$th root of $\left( v_{m},A_{1}u_{1}\right) $ and denote
it by $\left( v_{m},A_{1}u_{1}\right)^{1/m}$.  By \eqref{ValOflambda1^m} we can just reindexing the Puiseux series \eqref{MainResultsTheoremEigenvaluePuiseuxSeries} and \eqref{MainResultsTheoremEigenvectorPuiseuxSeries} and assume that
\begin{equation*}
\alpha_{1}=\left( -\frac{\frac{\partial }{\partial \varepsilon }\det
\left( \lambda I-A\left( \varepsilon \right) \right) |_{\left( \varepsilon
,\lambda \right) =\left( 0,\lambda _{0}\right) }}{\left( \frac{\frac{%
\partial ^{m}}{\partial \lambda ^{m}}\det \left( \lambda I-A\left(
\varepsilon \right) \right) |_{\left( \varepsilon ,\lambda \right) =\left(
0,\lambda _{0}\right) }}{m!}\right) }\right) ^{1/m}\text{.}
\end{equation*}%
\qquad

Next, we wish to prove that we can choose the perturbed eigenvectors (\ref%
{MainResultsTheoremEigenvectorPuiseuxSeries}) to satisfy the normalization
conditions (\ref{MainResultsTheoremNormalizationCondition}). But this follows by Theorem \ref{Generic Condition Theorem} (iv) and the fact $\beta_0$ is an eigenvector of $A(0)$ corresponding to the eigenvalue $\lambda_0$ since then $\left( v_{1},\beta_{0}\right) \not=0$ and so we may take $\frac{x_{h}\left( \varepsilon \right) }{\left(v_{1},x_{h}\left( \varepsilon \right) \right) }$, for $h=0, \ldots , m-1$, to be the perturbed eigenvectors in \eqref{MainResultsTheoremEigenvectorPuiseuxSeries} that satisfy the normalization conditions (\ref{MainResultsTheoremNormalizationCondition}). 

Now we are ready to begin showing that $\{\alpha_{s}\}_{s=1}^{\infty},\{\beta_{s}\}_{s=0}^{\infty}$ are given by the recursive formulas (\ref%
{MainResultsTheoremRecursiveFormulas1})-(\ref%
{MainResultsTheoremRecursiveFormulas3}). The first key step is proving the following:
\begin{eqnarray}
\label{ZerothKeyEq}
\left( A_{0}-\lambda _{0}I\right) \beta_{s}=-\tsum_{k=1}^{s}\left( A_{\frac{k}{m}%
}-\alpha_{k}I\right) \beta_{s-k}\text{, for }s\geq 1\text{,} \\
\beta_{0}=u_{1}\text{, }\beta_{s}=\Lambda \left( A_{0}-\lambda _{0}I\right) \beta_{s}%
\text{, for }s\geq 1\text{,}  \label{FirstKeyStepThm3.1Proof}
\end{eqnarray}
where we define $A_{\frac{k}{m}}:=0$, if $\frac{k}{m}%
\not\in 
\mathbb{N}
$.

The first equality holds since in a neighborhood of the origin
\begin{eqnarray*}
0 &=&\left( A\left( \varepsilon \right) -\lambda _{0}\left( \varepsilon
\right) I\right) x_{0}\left( \varepsilon \right) =\tsum\limits_{s=0}^{\infty }\left( \tsum\limits_{k=0}^{s}\left( A_{\frac{k}{%
m}}-\alpha_{k}I\right) \beta_{s-k}\right) \varepsilon ^{\frac{s}{m}}\text{.}
\end{eqnarray*}%

The second equality will be proven once we show $\beta_0=u_1$ and $\beta_s\in S:=\func{span}\{Ue_{i}|2\leq i\leq n\}$, for $s\geq 1$, where $U$ is the matrix from $\eqref{TheJordanNormalFormOfANot}$.  This will prove $\eqref{FirstKeyStepThm3.1Proof}$ because $\Lambda \left( A_{0}-\lambda _{0}I\right)$ acts as the identity on $S$ by Proposition \ref{Appendix A Proposition}.\ref{eq:A.1.iii}.  But these follow from the facts that $S=\{x\in\mathbb{C}^{n\times1}|(v_1,x)=0\}$ and the normalization conditions \eqref{MainResultsTheoremNormalizationCondition} imply that $\left( v_{1},\beta_{0}\right) =1$ and $\left(v_{1},\beta_{s}\right) =0$, for $s\geq 1$.

The next key step in this proof is the following lemma:
\begin{lemma}
For all $s\geq 0$ the following identity holds%
\begin{equation}
\left( A_{0}-\lambda _{0}I\right) \beta_{s}=\left\{ 
\begin{array}{c}
\sum\limits_{i=0}^{s}p_{s,i}u_{i}\text{, for }0\leq s\leq m-1 \\ 
\sum\limits_{i=0}^{m}p_{s,i}u_{i}-\sum\limits_{j=0}^{s-m}\sum%
\limits_{k=0}^{j}\sum\limits_{l=1}^{\left\lfloor \frac{s-j}{m}\right\rfloor
}p_{j,k}\Lambda ^{k}A_{l}\beta_{s-j-lm}\text{, for }s\geq m%
\end{array}%
\right.  \label{SecondKeyStepThm3.1Proof}
\end{equation}%
where we define $u_{0}:=0$.
\end{lemma}
\begin{proof}
The proof is by induction on $s$.  The statement is true for $s=0$ since $p_{0,0}u_{0}=0=\left( A_{0}-\lambda _{0}I\right) \beta_{0}$.  Now suppose it
was true for all $r$ with $0\leq r\leq s$ for some nonnegative integer $s$.  We will show the statement is true for $s+1$ as well.

Suppose $s+1\leq m-1$ then $\left( A_{0}-\lambda _{0}I\right)
\beta_{r}=\tsum_{i=0}^{r}p_{r,i}u_{i}$ for $0\leq r\leq s$ and we must show that 
$\left( A_{0}-\lambda _{0}I\right) \beta_{s+1}=\tsum_{i=0}^{s+1}p_{s+1,i}u_{i}$.  Well, for $1\leq r\leq s$,
\begin{eqnarray}
\beta_{r}\overset{\eqref{FirstKeyStepThm3.1Proof}}{=} \Lambda \left(
A_{0}-\lambda _{0}I\right) \beta_{r}=\tsum_{i=0}^{r}p_{r,i}\Lambda
u_{i}\overset{(\text{\ref{eq:A.1.iv}})}{=}\tsum_{i=1}^{r}p_{r,i}u_{i+1}\text{.} \label{x_r1stIndStep}
\end{eqnarray}
Hence the statement is true if $s+1\leq m-1$ since 
\begin{eqnarray*}
&&\left( A_{0}-\lambda _{0}I\right) \beta_{s+1} \overset{\eqref{ZerothKeyEq}}{=}-\tsum_{k=1}^{s+1}\left( A_{%
\frac{k}{m}}-\alpha_{k}I\right) \beta_{s+1-k}
\overset{\eqref{x_r1stIndStep}}{=}\tsum_{k=1}^{s+1}\tsum_{i=0}^{s+1-k}\alpha_{k}p_{s+1-k,i}u_{i+1}\\ 
&&\overset{\eqref{eq:DoubleSumId1}}{=} \tsum_{i=0}^{s}\left(
\tsum_{k=1}^{s+1-i}\alpha_{k}p_{s+1-k,i}\right) u_{i+1}
\overset{(\text{\ref{eq:B.1.vi}})}{=}\tsum_{i=0}^{s+1}p_{s+1,i}u_{i}\text{.}
\end{eqnarray*}%

Now suppose that $s+1\geq m$.  The proof is similar to what we just proved.
By the induction hypothesis \eqref{SecondKeyStepThm3.1Proof} is true for $1\leq r\leq s$ and $\beta_{r}\overset{\eqref{FirstKeyStepThm3.1Proof}}{=} \Lambda \left(
A_{0}-\lambda _{0}I\right) \beta_{r}$ thus
\begin{eqnarray}
\beta_{r}\overset{(\text{\ref{eq:A.1.iv}})}{=}\left\{ 
\begin{array}{c}
\tsum\limits_{i=0}^{r}p_{r,i}u_{i+1}\text{, for }0\leq r\leq m-1 \\ 
\sum\limits_{i=0}^{m-1}p_{r,i}u_{i+1}-\sum\limits_{j=0}^{r-m}\sum%
\limits_{k=0}^{j}\sum\limits_{l=1}^{\left\lfloor \frac{r-j}{m}\right\rfloor
}p_{j,k}\Lambda ^{k+1}A_{l}\beta_{r-j-lm}\text{, for }r\geq m \text{.}
\end{array}%
\right. \label{x_r2stIndStep}
\end{eqnarray}%
Hence we have 
\begin{eqnarray*}
\left( A_{0}-\lambda _{0}I\right) \beta_{s+1} &\overset{\eqref{ZerothKeyEq}}{=}&-\tsum_{k=1}^{s+1}\left( A_{%
\frac{k}{m}}-\alpha_{k}I\right) \beta_{s+1-k} \\
&\overset{\eqref{x_r2stIndStep}}{=}&-\tsum\limits_{l=1}^{\left\lfloor \frac{s+1}{m}\right\rfloor
}A_{l}\beta_{s+1-lm}+\tsum\limits_{k=1}^{s+1-m}\tsum\limits_{i=0}^{m-1}\alpha_{k}p_{s+1-k,i}u_{i+1} \\
&&-\tsum\limits_{k=1}^{s+1-m}\tsum\limits_{j=0}^{s+1-k-m}\tsum%
\limits_{i=0}^{j}\tsum\limits_{l=1}^{\left\lfloor \frac{s+1-k-j}{m}%
\right\rfloor }\alpha_{k}p_{j,i}\Lambda ^{i+1}A_{l}\beta_{s+1-k-j-lm} \\
&&+\tsum\limits_{k>s+1-m}^{s+1}\tsum\limits_{i=0}^{s+1-k}\alpha_{k}p_{s+1-k,i}u_{i+1} \\
&\overset{\eqref{eq:DoubleSumId1}}{=}&-\tsum\limits_{l=1}^{\left\lfloor \frac{s+1}{m}\right\rfloor
}A_{l}\beta_{s+1-lm}+\tsum\limits_{i=0}^{m-1}\left(
\tsum\limits_{k=1}^{s+1-i}\alpha_{k}p_{s+1-k,i}\right) u_{i+1} \\
&&-\tsum\limits_{k=1}^{s+1-m}\tsum\limits_{j=0}^{s+1-k-m}\tsum%
\limits_{i=0}^{j}\tsum\limits_{l=1}^{\left\lfloor \frac{s+1-k-j}{m}%
\right\rfloor }\alpha_{k}p_{j,i}\Lambda ^{i+1}A_{l}\beta_{s+1-k-j-lm} \\
&\overset{(\text{\ref{eq:B.1.vi})}}{=}&-\tsum\limits_{l=1}^{\left\lfloor \frac{s+1}{m}\right\rfloor
}A_{l}\beta_{s+1-lm}+\tsum\limits_{i=0}^{m}p_{s+1,i}u_{i} \\
&&-\tsum\limits_{k=1}^{s+1-m}\tsum\limits_{j=0}^{s+1-k-m}\tsum%
\limits_{i=0}^{j}\tsum\limits_{l=1}^{\left\lfloor \frac{s+1-k-j}{m}%
\right\rfloor }\alpha_{k}p_{j,i}\Lambda ^{i+1}A_{l}\beta_{s+1-k-j-lm}\text{.}
\end{eqnarray*}
Now let $a_{k,j,i}:=\sum\limits_{l=1}^{\left\lfloor \frac{s+1-k-j}{m}%
\right\rfloor }\alpha_{k}p_{j,i}\Lambda ^{i+1}A_{l}\beta_{s+1-k-j-lm}$.  Then using the sum identity 
\begin{eqnarray*}
&&\tsum\limits_{k=1}^{s+1-m}\tsum\limits_{j=0}^{s+1-k-m}\tsum%
\limits_{i=0}^{j}a_{k,j,i}
\overset{\eqref{eq:DoubleSumId1}}{=}\tsum\limits_{j=0}^{s-m}\tsum\limits_{k=1}^{s+1-j-m}\tsum%
\limits_{i=0}^{j}a_{k,j,i}
=\tsum\limits_{j=0}^{s-m}\tsum\limits_{i=0}^{j}\tsum%
\limits_{k=1}^{s+1-j-m}a_{k,j,i} \\
&&\overset{\eqref{eq:DoubleSumId3}}{=}\tsum\limits_{i=0}^{s-m}\tsum\limits_{j=i}^{s-m}\tsum%
\limits_{k=1}^{s+1-j-m}a_{k,j,i}
\overset{\eqref{eq:DoubleSumId4}}{=}\tsum\limits_{i=0}^{s-m}\tsum\limits_{q=i+1}^{s+1-m}\tsum%
\limits_{k=1}^{q-i}a_{k,q-k,i}
\end{eqnarray*} 
we can concluded that%
\begin{eqnarray*}
\left( A_{0}-\lambda _{0}I\right) \beta_{s+1}
&=&-\tsum\limits_{l=1}^{\left\lfloor \frac{s+1}{m}\right\rfloor
}A_{l}\beta_{s+1-lm}+\tsum\limits_{i=0}^{m}p_{s+1,i}u_{i} \\
&&-\tsum\limits_{i=0}^{s-m}\tsum\limits_{q=i+1}^{s+1-m}\tsum%
\limits_{k=1}^{q-i}\tsum\limits_{l=1}^{\left\lfloor \frac{s+1-q}{m}%
\right\rfloor }\alpha_{k}p_{q-k,i}\Lambda ^{i+1}A_{l}\beta_{s+1-q-lm} \\
&=&-\tsum\limits_{l=1}^{\left\lfloor \frac{s+1}{m}\right\rfloor
}A_{l}\beta_{s+1-lm}+\tsum\limits_{i=0}^{m}p_{s+1,i}u_{i} \\
&&-\tsum\limits_{i=0}^{s-m}\tsum\limits_{q=i+1}^{s+1-m}\tsum%
\limits_{l=1}^{\left\lfloor \frac{s+1-q}{m}\right\rfloor }\left(
\tsum\limits_{k=1}^{q-i}\alpha_{k}p_{q-k,i}\right) \Lambda
^{i+1}A_{l}\beta_{s+1-q-lm} \\
&\overset{(\text{\ref{eq:B.1.vi}})}{=}&-\tsum\limits_{l=1}^{\left\lfloor \frac{s+1}{m}\right\rfloor
}A_{l}\beta_{s+1-lm}+\tsum\limits_{i=0}^{m}p_{s+1,i}u_{i} \\
&&-\tsum\limits_{i=0}^{s-m}\tsum\limits_{q=i+1}^{s+1-m}\tsum%
\limits_{l=1}^{\left\lfloor \frac{s+1-q}{m}\right\rfloor }p_{q,i+1}\Lambda
^{i+1}A_{l}\beta_{s+1-q-lm} \\
&\overset{\eqref{eq:DoubleSumId2}}{=}&-\tsum\limits_{l=1}^{\left\lfloor \frac{s+1}{m}\right\rfloor
}A_{l}\beta_{s+1-lm}+\tsum\limits_{i=0}^{m}p_{s+1,i}u_{i} \\
&&-\tsum_{q=1}^{s+1-m}%
\tsum_{i=0}^{q-1}\tsum%
\limits_{l=1}^{\left\lfloor \frac{s+1-q}{m}\right\rfloor }p_{q,i+1}\Lambda
^{i+1}A_{l}\beta_{s+1-q-lm}\\
&\overset{\eqref{pPolysValWRightIndZero}}{=}&\tsum\limits_{i=0}^{m}p_{s+1,i}u_{i}-\tsum\limits_{j=0}^{s+1-m}\tsum%
\limits_{k=0}^{j}\tsum\limits_{l=1}^{\left\lfloor \frac{s+1-j}{m}%
\right\rfloor }p_{j,k}\Lambda ^{k}A_{l}\beta_{s+1-j-lm}\text{.}
\end{eqnarray*}
But this is the statement we needed to prove for $s+1\geq m$.  Therefore by
induction the statement (\ref{SecondKeyStepThm3.1Proof}) is true for all $%
s\geq 0$ and the lemma is proved.
\qquad \end{proof}

The lemma above is the key to prove the recursive formulas for $\alpha_{s}$
and $\beta_{s}$ as given by (\ref{MainResultsTheoremRecursiveFormulas1})-(\ref%
{MainResultsTheoremRecursiveFormulas3}).  First we prove that $\beta_{s}$ is
given by (\ref{MainResultsTheoremRecursiveFormulas3}).  For $s=0$ we have
already shown $\beta_{0}=u_{1}=p_{0,0}u_{1}$.  So suppose $s\geq 1$.  Then by \eqref{FirstKeyStepThm3.1Proof} and \eqref{SecondKeyStepThm3.1Proof} we find that
\begin{eqnarray*}
\beta_{s}&\overset{(\text{\ref{eq:A.1.iv}})}{=}&\left\{ 
\begin{array}{c}
\sum\limits_{i=0}^{s}p_{s,i}u_{i+1}\text{, if }0\leq s\leq m-1 \\ 
\sum\limits_{i=0}^{m-1}p_{s,i}u_{i+1}-\sum\limits_{j=0}^{s-m}\sum%
\limits_{k=0}^{j}\sum\limits_{l=1}^{\left\lfloor \frac{s-j}{m}\right\rfloor
}p_{j,k}\Lambda ^{k+1}A_{l}\beta_{s-j-lm}\text{, if }s\geq m  \text{.}%
\end{array}%
\right.
\end{eqnarray*}%
This proves that $\beta_{s}$ is given by (\ref%
{MainResultsTheoremRecursiveFormulas3}).

Next we will prove that $\alpha_{s}$ is given by (\ref%
{MainResultsTheoremRecursiveFormulas1}) and (\ref%
{MainResultsTheoremRecursiveFormulas2}).  We start with $s=1$ and prove $%
\alpha_{1}$ is given by (\ref{MainResultsTheoremRecursiveFormulas1}).  First, $(A_{0}-\lambda _{0}I)^{\ast }v_{m}=0$ and $(v_m,u_i)=\delta_{m,i}$ hence
\begin{eqnarray*}
0=\left( v_{m},(A_{0}-\lambda _{0}I)\beta_{m}\right) \overset{\eqref{SecondKeyStepThm3.1Proof}}{=}\left( v_{m},\tsum\limits_{i=0}^{m}p_{m,i}u_{i}-A_{1}u_{1}\right)
\overset{(\text{\ref{eq:B.1.iv}})}{=} \alpha_{1}^{m}-\left( v_{m},A_{1}u_{1}\right)
\end{eqnarray*}%
so that $\alpha_{1}^{m}=\left( v_{m},A_{1}u_{1}\right) $.  This and identity \eqref{ValOflambda1^m} imply that formula (\ref%
{MainResultsTheoremRecursiveFormulas1}) is true.

Finally, suppose that $s\geq 2$.  Then $(A_{0}-\lambda _{0}I)^{\ast }v_{m}=0$ and $(v_m,u_i)=\delta_{m,i}$ implies
\begin{eqnarray*}
0 &=&\left( v_{m},(A_{0}-\lambda _{0}I)\beta_{m+s-1}\right) \\
&\overset{\eqref{SecondKeyStepThm3.1Proof}}{=}&\left(
v_{m},\tsum\limits_{i=0}^{m}p_{m+s-1,i}u_{i}-\tsum\limits_{j=0}^{
s-1}\tsum\limits_{k=0}^{j}\tsum\limits_{l=1}^{\left\lfloor \frac{%
m+s-1 -j}{m}\right\rfloor }p_{j,k}\Lambda ^{k}A_{l}\beta_{
m+s-1-j-lm}\right) \\
&\overset{\eqref{eq:DoubleSumId3}}{=}&p_{m+s-1,m}-\tsum\limits_{k=0}^{s-1}\tsum\limits_{j=k}^{s-1}p_{j,k}\left(
\left( \Lambda ^{\ast }\right) ^{k}v_{m},\tsum\limits_{l=1}^{\left\lfloor 
\frac{m+s-1-j}{m}\right\rfloor }A_{l}\beta_{m+s-1-j-lm}\right)  \\
&\overset{(\text{\ref{eq:B.1.ii}})}{=} &r_{s-1}+m\alpha_{1}^{m-1}\alpha_{s}-\tsum\limits_{i=0}^{s-1}\tsum\limits_{j=i}^{s-1}p_{j,i}\left(
\left( \Lambda ^{\ast }\right) ^{i}v_{m},\tsum\limits_{k=1}^{\left\lfloor \frac{m+s-1-j}{m}\right\rfloor
}A_{k}\beta_{m+s-1-j-km}\right)
\end{eqnarray*}%
Therefore with this equality, the fact $\alpha_{1}\neq 0$, and Proposition \ref{Appendix A Proposition}.\ref{eq:A.1.v}, we can
solve for $\alpha_{s}$ and we will find that it is given by (\ref%
{MainResultsTheoremRecursiveFormulas2}). This completes the proof.
\qquad \endproof

\subsection{Proof of Corollaries \protect\ref{Calculation kth order by
matrices corollary} and \protect\ref{Second order coefficients corollary}}

Both corollaries follow almost trivially now.  To prove Corollary \ref{MurdockClarkMyResponseCorollary}, we just examine the
recursive formulas (\ref%
{MainResultsTheoremRecursiveFormulas1})-(\ref%
{MainResultsTheoremRecursiveFormulas3}) in Theorem \ref{Main Results Theorem} to see that $\alpha _{k}, \beta_{k} $ requires only $A_{0}$, \ldots , $A_{\left\lfloor \frac{m+k-1}{m}%
\right\rfloor }$.  To prove Corollary \ref{2ndOrderCoeffCorollary}, we use Proposition \ref{Appendix B Proposition} to show that $$p_{0,0}=1, p_{1,0}=p_{2,0}=0, p_{1,1}=\alpha_1, p_{2,1}=\alpha_2, p_{2,2}=\alpha_1^2$$ and then from this and (\ref%
{MainResultsTheoremRecursiveFormulas1})-(\ref%
{MainResultsTheoremRecursiveFormulas3}) we get the desired result for $\alpha_1, \alpha_2, \beta_0, \beta_1, \beta_2$ in terms of $A_0, A_1, A_2$.  The last part to prove is the formula for $\alpha_{2}$ in terms of $f\left( \varepsilon ,\lambda \right)$ and its partial derivatives.  But the formula follows from the series representation of $f\left( \varepsilon ,\lambda \right) $ in (%
\ref{SeriesRepresentationForf}) and $\lambda _{0}\left( \varepsilon \right) $ in
(\ref{MainResultsTheoremEigenvaluePuiseuxSeries}) since, for $\varepsilon$ in a neighborhood of the origin,
\begin{eqnarray*}
0 &=&f\left( \varepsilon ,\lambda _{0}\left( \varepsilon \right) \right)\\
 &=& \left( a_{10}+a_{0m}p_{m.m}\right) \varepsilon \\
&&+\left\{ 
\begin{array}{c}
\left( a_{01}p_{2,1}+a_{02}p_{2,2}+a_{11}p_{1,1}+a_{20}p_{00}\right)
\varepsilon ^{2}\text{, for }m=1 \\ 
\left( a_{0m}p_{m+1,m}+a_{0m+1}p_{m+1,m+1}+a_{11}p_{1,1}\right) \varepsilon
^{\frac{m+1}{m}}\text{, for }m>1%
\end{array}%
\right. \\
&&+O\left( \varepsilon ^{\frac{m+2}{m}}\right)
\end{eqnarray*}%
which together with Proposition \ref{Appendix B Proposition} implies the formula for $\alpha_{2}$.
\qquad \endproof

\appendix\section*{}%

The fundamental properties of the matrix $\Lambda$ defined in \eqref{PartialInverseOfTheJordanNormalFormOfANot} which are needed in this paper are given in the following proposition: 
\begin{proposition}
\label{Appendix A Proposition}

\begin{enumerate}[label=\textnormal{(\roman{*})}, ref=\roman{*}]
\item \label{eq:A.1.iii} We have $\Lambda \left( A_{0}-\lambda _{0}I\right) Ue_{1}=0$, $\Lambda
\left( A_{0}-\lambda _{0}I\right) Ue_{i}=Ue_{i}$, for $2\leq i\leq n$,

\item For $1\leq i\leq m-1$ we have 
\begin{eqnarray}
\Lambda u_{m}=0\text{, } \Lambda u_{i}=u_{i+1} \label{eq:A.1.iv}
\end{eqnarray}

\item \label{eq:A.1.v} $\Lambda ^{\ast }v_{1}=0$, and $\Lambda ^{\ast }v_{i}=v_{i-1}$, for $%
2\leq i\leq m$.
\end{enumerate}
\end{proposition}

{\em Proof}.
i.\  Using the fact $J_{m}\left( 0\right)
^{\ast }J_{m}\left( 0\right) =\diag[0,I_{m-1}]$, \eqref{TheJordanNormalFormOfANot}, and \eqref{PartialInverseOfTheJordanNormalFormOfANot} we find by block multiplication that $U^{-1}\Lambda \left( A_{0}-\lambda _{0}I\right) U=\diag[0,I_{n-1}]$.  This implies the result.

ii.\ \& iii.\  The results follow from the definition of $u_i, v_i$ in \eqref{BiorthogonalVectors}, \eqref{BiorthogonalVectors_v} and the fact
\begin{eqnarray*}
(U^{-1}\Lambda U)^*=U^{\ast }\Lambda ^{\ast }\left( U^{-1}\right) ^{\ast }=\left[ 
\begin{array}{c|c}
J_{m}\left( 0\right) &  \\ \hline
& \left[ \left( W_{0}-\lambda _{0}I_{n-m}\right) ^{-1}\right] ^{\ast }%
\end{array}%
\right] \text{.} \qquad \endproof
\end{eqnarray*}

\section*{}%

This appendix contains two propositions. The first proposition gives fundamental identities that help to characterize the polynomials $\{p_{j,i}\}_{j=i}^\infty$ and $\left\{r_l\right\}_{l\in\mathbb{N}}$ in \eqref{pPolysValWRightIndZero} and \eqref{rPolysVal}.  The second proposition gives explicit recursive formulas to calculate these polynomials.

We may assume $\sum_{j=1}^{\infty }\alpha_{j}z^{j}$ is a convergent Taylor
series and $\alpha_1\not=0$.

\begin{proposition}
\label{Appendix B Proposition}The polynomials $\{p_{j,i}\}_{j=i}^\infty$ and $\left\{r_l\right\}_{l\in\mathbb{N}}$ have the
following properties:

\begin{enumerate}[label=\textnormal{(\roman{*})}, ref=\roman{*}]
\item \label{eq:B.1.i} $\sum\limits_{j=i}^{\infty }p_{j,i}z^{j}=\left(
\sum\limits_{j=1}^{\infty }\alpha_{j}z^{j}\right) ^{i}$, for $j\geq i\geq
0 $.

\item\label{PropB1ii} For $l\geq 1$ we have
\begin{eqnarray}
r_{l}=p_{m+l,m}-m\alpha_{1}^{m-1}\alpha_{l+1}. \label{eq:B.1.ii} 
\end{eqnarray}

\item \label{eq:B.1.iii} $p_{j,1}=\alpha_{j}$, for $j\geq 1$.

\item For $j\geq 0$ we have
\begin{eqnarray}
p_{j,j}=\alpha_{1}^{j}. \label{eq:B.1.iv}
\end{eqnarray}

\item \label{eq:B.1.v} $p_{j+1,j}=j\alpha_{1}^{j-1}\alpha_{2}$, for $j>0$.

\item
For $j\geq i>0$ we have
\begin{eqnarray}
\sum\limits_{q=1}^{j-i+1}\alpha_{q}p_{j-q,i-1}=p_{j,i}. \label{eq:B.1.vi}
\end{eqnarray}
\end{enumerate}
\end{proposition}

{\em Proof}.
i.\  For $i\geq 0$, $\left(\sum\limits_{j=1}^{\infty }\alpha_{j}z^{j}\right) ^{i}=\sum\limits_{s_{1}=1}^{\infty }\cdots \sum\limits_{s_{i}=1}^{\infty }\left(
\prod\limits_{\varrho =1}^{i}\alpha_{s_{\varrho }}\right) z^{s_{1}+\cdots
+s_{i}}=\sum\limits_{j=i}^{\infty }p_{j,i}z^{j}\text{.}$

ii.\  Let $l\geq 1$.  Then by definition of $p_{m+l,m}$ we have 
\begin{eqnarray*}
p_{m+l,m} &=&\sum_{\substack{ s_{1}+\cdots +s_{m}=m+l  \\ 1\leq s_{\varrho }\leq l+1 \\ \exists \varrho \in \left\{ 1,...,m\right\} \text{ such that }s_{\varrho
}=l+1\text{ }}}\prod\limits_{\varrho =1}^{m}\alpha_{s_{\varrho }}+\sum 
_{\substack{ s_{1}+\cdots +s_{m}=m+l  \\ 1\leq s_{\varrho }\leq l+1  \\ \NEG%
{\exists}\varrho \in \left\{ 1,...,m\right\} \text{ such that }s_{\varrho
}=l+1\text{ }}}\prod\limits_{\varrho =1}^{m}\alpha_{s_{\varrho }} \\
&=&m\alpha_{1}^{m-1}\alpha_{l+1}+r_{l}\text{.}
\end{eqnarray*}

iii.\  For $j\geq 1$ we have $p_{j,1}=\dsum 
_{\substack{ s_{1}=j  \\ 1\leq s_{\varrho }\leq j}}\prod_{\varrho
=1}^{1}\alpha_{s_{\varrho }}=\alpha_{j}$.

iv.\  For $j\geq 0$, $p_{0,0}=1$ and $p_{j,j}=\dsum_{\substack{ s_{1}+\cdots +s_{j}=j 
\\ 1\leq s_{\varrho }\leq 1}}\prod_{\varrho =1}^{j}\alpha_{s_{\varrho
}}=\prod_{\varrho =1}^{j}\alpha_{1}=\alpha_{1}^{j}$.

v.\  For $j>0$, $
p_{j+1,j}=\sum_{\substack{ s_{1}+\cdots +s_{j}=j+1  \\ 1\leq s_{\varrho
}\leq 2}}\prod_{\varrho =1}^{j}\alpha_{s_{\varrho }}=\tsum_{\varrho
=1}^{j}\alpha_{1}^{j-1}\alpha_{2}=j\alpha_{1}^{j-1}\alpha_{2}\text{.}$

vi.\  It follows by $$\sum\limits_{j=i}^{\infty }p_{j,i}z^{j}\overset{\textrm{(i)}}{=}\sum\limits_{j=i-1}^{\infty }p_{j,i-1}z^{j}\sum\limits_{j=1}^{\infty
}p_{j,1}z^{j}=\sum\limits_{j=i}^{\infty }\left(
\sum\limits_{q=1}^{j-i+1}p_{j-q,i-1}p_{q,1}\right) z^{j}\text{.} \qquad \endproof$$

This next proposition gives explicit recursive formulas to calculate the polynomials $\{p_{j,i}\}_{j=i}^\infty$ and $\left\{r_l\right\}_{l\in\mathbb{N}}$.
\begin{proposition}\label{ExplicitRecursiveFormulasPolynomials}
For each $i\geq 0$, the sequence of polynomials, $\{p_{j,i}\}_{j=i}^\infty$, is given by the recursive formula
\begin{eqnarray}
\label{pPolysRecursFormula}
p_{i,i}=\alpha_1^i\text{, } p_{j,i}=\frac{1}{(j-i)\alpha_1}\sum_{k=i}^{j-1}[(j+1-k)i-k]\alpha_{j+1-k}p_{k,i}
\text{, for }j>i.
\end{eqnarray}
Furthermore, the polynomials $\left\{r_l\right\}_{l\in\mathbb{N}}$ are given by the recursive formula:
\begin{eqnarray}
r_1=0\text{, }r_l&=&\frac{1}{l\alpha_1}\sum_{j=1}^{l-1}[(l+1-j)m-(m+j)]\alpha_{l+1-j}r_j+ \label{rPolysRecursFormula} \\
&& \frac{m}{l}\alpha_1^{m-2}\sum_{j=1}^{l-1}[(l+1-j)m-(m+j)]\alpha_{l+1-j}\alpha_{j+1}
\text{, for }l>1. \notag
\end{eqnarray}
\end{proposition}
\begin{proof}
We begin by showing \eqref{pPolysRecursFormula} is true. For $i=0$, \eqref{pPolysRecursFormula} follows from the definition of the $p_{j,0}$. If $i>0$ then by \eqref{eq:B.1.iv} and \cite[(1.1) \& (3.2)]{Gould1974} it follows that \eqref{pPolysRecursFormula} is true.  Lets now prove \eqref{rPolysRecursFormula}. From \eqref{eq:B.1.ii} and \eqref{pPolysRecursFormula}, we have
\begin{eqnarray*}
r_l&=&p_{m+l,m}-m\alpha_1^{m-1}\alpha_{l+1}\\
&=&\frac{1}{l\alpha_1}\sum_{k=m}^{m+l-1}[(m+l+1-k)m-k]\alpha_{m+l+1-k}p_{k,m}-m\alpha_1^{m-1}\alpha_{l+1}\\
&\overset{(\text{\ref{eq:B.1.iv}})}{=}&\frac{1}{l\alpha_1}\sum_{k=m+1}^{m+l-1}[(m+l+1-k)m-k]\alpha_{m+l+1-k}p_{k,m}\\
&=&\frac{1}{l\alpha_1}\sum_{j=1}^{l-1}[(l+1-j)m-(m+j)]\alpha_{l+1-j}p_{m+j,m}\\
&\overset{(\text{\ref{eq:B.1.ii}})}{=}&\frac{1}{l\alpha_1}\sum_{j=1}^{l-1}[(l+1-j)m-(m+j)]\alpha_{l+1-j}\left(r_j+m\alpha_1^{m-1}\alpha_{j+1}\right)\\
&=&\frac{1}{l\alpha_1}\sum_{j=1}^{l-1}[(l+1-j)m-(m+j)]\alpha_{l+1-j}r_j+ \\
&& \frac{m}{l}\alpha_1^{m-2}\sum_{j=1}^{l-1}[(l+1-j)m-(m+j)]\alpha_{l+1-j}\alpha_{j+1}\text{,}
\end{eqnarray*}
for $l > 1$. This completes the proof.
\qquad \end{proof}
\section*{}
These double sum identities are used in the proof of Theorem 3.1
\begin{eqnarray}
\label{eq:DoubleSumId1} \sum_{x=c}^d \sum_{y=0}^{d-x}a_{x,y}=\sum_{y=0}^{d-c} \sum_{x=c}^{d-y}a_{x,y}\text{,} \\
\label{eq:DoubleSumId2} \sum_{x=0}^{d-1} \sum_{y=x+1}^{d}a_{x,y}= \sum_{y=1}^d \sum_{x=0}^{y-1}a_{x,y}\text{,}\\
\label{eq:DoubleSumId3} \sum_{x=0}^d \sum_{y=0}^{x}a_{x,y}=\sum_{y=0}^d \sum_{x=y}^{d}a_{x,y}\text{,}\\
\label{eq:DoubleSumId4} \sum_{y=c}^{d-1} \sum_{x=1}^{d-y}a_{x,y}=\sum_{q=c+1}^{d} \sum_{x=1}^{q-c}a_{x,q-x}\text{.}
\end{eqnarray}

\thanks{\textbf{Acknowledgments.}}  I would like to thank Prof. Alexander Figotin for bringing the subject of this paper to my attention and for all the helpful suggestions and encouragement given in the various stages of writing this paper.  I am also indebted to the anonymous referees for their valuable comments on my original manuscript.%


\begin{thebibliography}{27}

\bibitem{Baumgartel1985}
{\sc H.~Baumg{\"a}rtel}, {\em Analytic perturbation theory for matrices and
  operators}, vol.~15 of Operator Theory: Advances and Applications,
  Birkh\"auser Verlag, Basel, 1985.

\bibitem{Kato1995}
{\sc Tosio Kato}, {\em Perturbation theory for linear operators}, Classics in
  Mathematics, Springer-Verlag, Berlin, 1995.
\newblock Reprint of the 1980 edition.

\bibitem{Lidskii1966}
{\sc V.~B. Lidskii}, {\em Perturbation theory of non-conjugate operators}, USSR
  Computational Mathematics and Mathematical Physics, 6 (1966), pp.~73 -- 85.

\bibitem{MoroBurkeOverton1997}
{\sc Julio Moro, James~V. Burke, and Michael~L. Overton}, {\em On the
  {L}idskii-{V}ishik-{L}yusternik perturbation theory for eigenvalues of
  matrices with arbitrary {J}ordan structure}, SIAM J. Matrix Anal. Appl., 18
  (1997), pp.~793--817.

\bibitem{VainbergTrenogin1974}
{\sc M.~M. Va{\u\i}nberg and V.~A. Trenogin}, {\em Theory of branching of
  solutions of non-linear equations}, Noordhoff International Publishing,
  Leyden, 1974.
\newblock Translated from the Russian by Israel Program for Scientific
  Translations.

\bibitem{VishikLjusternik1960}
{\sc M.~I. Vi{\v{s}}ik and L.~A. Ljusternik}, {\em Solution of some
  perturbation problems in the case of matrices and self-adjoint or
  non-selfadjoint differential equations. {I}}, Russian Math. Surveys, 15
  (1960), pp.~1--73.

\bibitem{FigotinVitebskiy2006}
{\sc Alex Figotin and Ilya Vitebskiy}, {\em Slow light in photonic crystals}, Waves in Random and Complex Media, 16 (2006), pp.~293--382.

\bibitem{MumcuSertelVolakisVitebskiyFigotin2006}
{\sc G.~Mumcu, K.~Sertel, J.~L. Volakis, I.~Vitebskiy, and A.~Figotin}, {\em RF propagation in finite thickness unidirectional magnetic photonic crystals}, Antennas and Propagation, IEEE Transactions on, 53 (2006), pp.~4026--4034.

\bibitem{YargaSertelVolakis2008}
{\sc S.~Yarga, K.~Sertel, and J.~L. Volakis}, {\em Degenerate Band Edge Crystals for Directive Antennas}, Antennas and Propagation, IEEE Transactions on, 56 (2008), pp.~119--126.

\bibitem{LancasterTismenetsky1985}
{\sc Peter Lancaster and Miron Tismenetsky}, {\em The theory of matrices},
  Computer Science and Applied Mathematics, Academic Press Inc., Orlando, FL,
  second~ed., 1985.

\bibitem{AndrewChuLancaster1993}
{\sc Alan~L. Andrew, K.-w.~Eric Chu, and Peter Lancaster}, {\em Derivatives of
  eigenvalues and eigenvectors of matrix functions}, SIAM J. Matrix Anal.
  Appl., 14 (1993), pp.~903--926.

\bibitem{AndrewTan2000}
{\sc Alan~L. Andrew and Roger C.~E. Tan}, {\em Iterative computation of
  derivatives of repeated eigenvalues and the corresponding eigenvectors},
  Numer. Linear Algebra Appl., 7 (2000), pp.~151--167.


\bibitem{Chu1990}
{\sc King-wah~Eric Chu}, {\em On multiple eigenvalues of matrices depending on
  several parameters}, SIAM J. Numer. Anal., 27 (1990), pp.~1368--1385.

\bibitem{HrynivLancaster1999}
{\sc R.~Hryniv and P.~Lancaster}, {\em On the perturbation of analytic matrix
  functions}, Integral Equations Operator Theory, 34 (1999), pp.~325--338.

\bibitem{JeannerodPflugel1999}
{\sc C.-P. Jeannerod and E.~Pfl{\"u}gel}, {\em A reduction algorithm for
  matrices depending on a parameter}, in Proceedings of the 1999
  {I}nternational {S}ymposium on {S}ymbolic and {A}lgebraic {C}omputation
  ({V}ancouver, {BC}), New York, 1999, ACM, pp.~121--128 (electronic).

\bibitem{Lancaster1964}
{\sc P.~Lancaster}, {\em On eigenvalues of matrices dependent on a parameter},
  Numer. Math., 6 (1964), pp.~377--387.

\bibitem{LancasterMarkusZhou2003}
{\sc P.~Lancaster, A.~S. Markus, and F.~Zhou}, {\em Perturbation theory for
  analytic matrix functions: the semisimple case}, SIAM J. Matrix Anal. Appl.,
  25 (2003), pp.~606--626 (electronic).

\bibitem{LangerNajman1989}
{\sc H.~Langer and B.~Najman}, {\em Remarks on the perturbation of analytic
  matrix functions. {II}}, Integral Equations Operator Theory, 12 (1989),
  pp.~392--407.

\bibitem{MaEdelman1998}
{\sc Yanyuan Ma and Alan Edelman}, {\em Nongeneric eigenvalue perturbations of
  {J}ordan blocks}, Linear Algebra Appl., 273 (1998), pp.~45--63.

\bibitem{MeyerStewart1988}
{\sc Carl~D. Meyer and G.~W. Stewart}, {\em Derivatives and perturbations of
  eigenvectors}, SIAM J. Numer. Anal., 25 (1988), pp.~679--691.

\bibitem{MoroDopico2001}
{\sc Julio Moro and Froil{\'a}n~M. Dopico}, {\em First order eigenvalue
  perturbation theory and the {N}ewton diagram}, in Applied mathematics and
  scientific computing ({D}ubrovnik, 2001), Kluwer/Plenum, New York, 2003,
  pp.~143--175.

\bibitem{SeyranianMailybaev2003}
{\sc A.~P. Seyranian and A.~A. Mailybaev}, {\em Multiparameter stability theory
  with mechanical applications}, vol.~13 of Series on Stability, Vibration and
  Control of Systems. Series A: Textbooks, Monographs and Treatises, World
  Scientific Publishing Co. Inc., River Edge, NJ, 2003.

\bibitem{Sun1985}
{\sc Ji~Guang Sun}, {\em Eigenvalues and eigenvectors of a matrix dependent on
  several parameters}, J. Comput. Math., 3 (1985), pp.~351--364.

\bibitem{Sun1990}
\leavevmode\vrule height 2pt depth -1.6pt width 23pt, {\em Multiple eigenvalue
  sensitivity analysis}, Linear Algebra Appl., 137/138 (1990), pp.~183--211.

\bibitem{Krantz1992}
{\sc Steven~G. Krantz}, {\em Function theory of several complex variables}, The
  Wadsworth \& Brooks/Cole Mathematics Series, Wadsworth \& Brooks/Cole
  Advanced Books \& Software, Pacific Grove, CA, second~ed., 1992.

\bibitem{IpsenRehman2008}
{\sc Ilse C.~F. Ipsen and Rizwana Rehman}, {\em Perturbation bounds for
  determinants and characteristic polynomials}, SIAM J. Matrix Anal. Appl., 30
  (2008), pp.~762--776.

\bibitem{Gould1974}
{\sc H.~W. Gould}, {\em Coefficient identities for powers of {T}aylor and
  {D}irichlet series}, Amer. Math. Monthly, 81 (1974), pp.~3--14.

\end{thebibliography}

\end{document}